\documentclass[12pt,reqno]{amsart}
\usepackage{amsmath,amssymb,enumerate,bbm,amsthm,graphicx,verbatim}
\usepackage[all,2cell,cmtip]{xy}
\usepackage{stmaryrd}
\title[POINCAR\'E DUALITY IN MORAVA $K$-THEORY FOR ORBIFOLDS]{POINCAR\'E DUALITY IN MORAVA $K$-THEORY FOR CLASSIFYING SPACES OF ORBIFOLDS }
\author{Man Chuen Cheng}
\thanks{The author is supported by the Croucher Foundation}
\address{Department of Mathematics, University of British Columbia, Vancouver, BC V6T 1Z2, Canada}
\email{mccheng@math.ubc.ca}

\newcommand{\tmop}[1]{\ensuremath{\operatorname{#1}}}

\newcommand{\R}{\mathbb{R}}
\newcommand{\SSS}{\mathbb{S}}
\newcommand{\mfX}{\mathfrak{X}}
\newcommand{\mfY}{\mathfrak{Y}}

\newcommand{\mfV}{\mathfrak{V}}

\newcommand{\bZ}{\mathbb{Z}}

\newcommand{\bC}{\mathbb{C}}
\newcommand{\Top}{\tmop{\textbf{Top}}}
\newcommand{\Diff}{\tmop{\textbf{Diff}}}

\newcommand{\aaa}{\alpha}
\newcommand{\bbb}{\beta}
\newcommand{\Ho}{\tmop{Ho}}

\newtheorem{theorem}{Theorem}[section]
\newtheorem{lemma}[theorem]{Lemma}
\newtheorem{corollary}[theorem]{Corollary}
\newtheorem{proposition}[theorem]{Proposition}

\theoremstyle{definition}
\newtheorem{definition}[theorem]{Definition}
\newtheorem{example}[theorem]{Example}

\newtheorem{remark}[theorem]{Remark}

\def\quotient#1#2{%
    \raise1ex\hbox{$#1$}\Big/\lower1ex\hbox{$#2$}%
}

\begin{document}

\begin{abstract}
In \cite{Greenlees} Greenlees and Sadofsky showed that the classifying spaces of finite groups are self-dual with respect to Morava $K$-theory $K(n)$. Their duality map was constructed using a transfer map. We generalize their duality map and prove a $K(n)$-version of Poincar\'{e} duality for classifying spaces of orbifolds. By regarding these classifying spaces as the homotopy types of certain differentiable stacks, our construction can be viewed as a stack version of Spanier-Whitehead type construction. Some examples and calculations will be given at the end.
\end{abstract}

\maketitle

\tableofcontents

%\begin{Acknowledgments}
%First and foremost, I would like to express my sincere thanks to my advisor S\o ren Galatius for his patient guidance and continuous encouragement throughout my graduate %studies. He has been always generous with his time and idea to discuss mathematics with me. I am grateful to work in this interesting thesis project suggested by him. 

%I am thankful to the Stanford Mathematics Department. I would like to say thanks to Ralph Cohen, Gunnar Carlsson and Andrew Blumberg for teaching me algebraic topology, and %to my friends Robin, Anssi, Eric, Tracy, Jonathan, Jeremy, Nathan, Martin, Jose, Jesse, Jeff, Daniel, Ken, Lan and my officemates for making my time at Stanford a very pleasant %and enjoyable experience.

%I would also like to express my gratitude to the Chinese University of Hong Kong where I had five great years for my undergraduate and master studies. I especially thank my %master thesis advisor Luen Fai Tam. He is a good teacher and has a big influence on my mathematics career.

%I also thank John Greenlees for sharing with me his viewpoint and idea on my thesis topic.

%Last but not least, I am deeply indebted to my family for their love and support throughout the years. I dedicate this thesis to them.

\section{Introduction and statement of results}

\subsection{Equivariant duality}
Manifolds and classifying spaces of groups
are two important classes of spaces. For a closed oriented
manifold $M$, Poincar\'{e} duality gives us a simple relation between its homology
and cohomology
\[ H_{\ast} (M, \mathbb{Z}) \cong H^{m - \ast} (M, \mathbb{Z}). \]
Poincar\'e duality can be viewed as a consequence of the Thom isomorphism and the Spanier-Whitehead duality
\begin{equation}
  \Sigma^{\infty} M_+ \simeq F (M^{- T M}, \mathbb{S}). \label{swd} 
\end{equation}
Here $F(M^{- T M}, \mathbb{S})$ denotes the function spectrum of maps from the Thom spectrum $M^{- T M}$ to the sphere spectrum $\SSS$.

Classifying spaces of groups, on the other hand, usually have non-zero
integral homology groups in infinitely many degrees and hence cannot satisfy a
duality of the form $H^{\ast} (B G, \mathbb{Z}) \cong H_{m - \ast} (B G,
\mathbb{Z})$. Nevertheless, they exhibit duality properties with respect to
Morava $K$-theory.

For each fixed prime $p$, Morava $K$-theory is a sequence of homology
theories $K (n), n \geqslant 0$, with coefficient ring $K (n)_{\ast} \cong
\mathbb{F}_p [v_n, v_n^{- 1}]$ and $\deg v_n = 2 (p^n - 1)$. In \cite{Ravenel} Ravenel
showed that for a finite group $G$, the $n$-th Morava $K$-theory cohomology
ring $K (n)^{\ast} (B G)$ has finite rank over $K (n)^{\ast}$. Later, it was
shown in \cite{Greenlees} that 

\begin{equation}
  K (n)_{\ast} (B G) \cong K (n)^{- \ast} (B G) \label{bgknd},
\end{equation}
which can be regarded as a $K(n)$-version of Poincar\'{e} duality for $BG$. Indeed, it was shown in \cite{Hovey} that the $K(n)$-localization of $\Sigma^{\infty}BG_+$ is self-dual in the category of $K(n)$-local spectra. Strickland \cite{Strickland} showed that (\ref{bgknd}) can be obtained using a transfer map for a covering map version of the diagonal $B G \rightarrow B G \times B G$. The transfer map, constructed by equivariant stable homotopy theory, gives a map of spectra
\begin{equation}
  \Sigma^{\infty} (B G \times B G)_+ \rightarrow \Sigma^{\infty} B G_+.
  \label{transfer}
\end{equation}
Composing this with the collapse map $\Sigma^{\infty} B G_+ \to
\Sigma^{\infty}pt_+=\mathbb{S}$ gives $\Sigma^{\infty} B G_+ \wedge \Sigma^{\infty} B G_+
\rightarrow \mathbb{S}$, with adjoint
\begin{equation}
   \Sigma^{\infty} B G_+ \rightarrow F (\Sigma^{\infty} B G_+,
  \mathbb{S}) \label{bgd} .
\end{equation}
Taking $K(n)$-homology gives us the $K(n)$-self-duality of $B G$ in (\ref{bgknd}).

%In view of the dualities (\ref{swd}) for manifolds and (\ref{bgknd}) for classifying spaces of groups, one may ask for a duality theorem for spaces of the form $EG\times_GM$, where $G$ is a compact Lie group and $M$ is a $G$-manifold. For instance, if $H$ is a finite subgroup of $G$, then $M=G/H$ is a closed $G$-manifold and $EG\times_GM\simeq BH$ is self dual with respect to Morava $K$-theory by (\ref{bgknd}). In general, suppose $G$ acts on a compact manifold $M$ smoothly. The action is said to be almost free if it has finite stabilizers. In such case, the $G$-action on $M$ defines an orbifold. Here we adopt a more general definition of orbifolds than the classical case and do not assume the $G$-action on $M$ to be effective. The classifying space of such an orbifold is $EG\times_GM$. One can generalize the result of Ravenel for classifying spaces of finite groups mentioned above to show that $K(n)^{\ast}(EG\times_GM)$ has finite rank over $K(n)^{\ast}$. This finiteness condition is necessary for $EG\times_GM$ to satisfy a $K(n)$-version of Poincar\'e duality. Our main theorem concerns the $K(n)$-duality of such spaces and thus can be viewed as a $K(n)$-version of Poincar\'e duality for classifying spaces of certain orbifolds. It is a conjecture that every orbifold is equivalent to a quotient orbifold arising from a smooth, almost free Lie group action on a smooth manifold. See \cite{Hen} for partial results towards this conjecture. These suggest that our main theorem covers a large class of interesting orbifolds. 

In view of the dualities (\ref{swd}) for manifolds and (\ref{bgknd}) for classifying spaces of groups, one may ask for a duality theorem for spaces of the form $EG\times_GM$, where $G$ is a compact Lie group and $M$ is a $G$-manifold. For instance, if $H$ is a finite subgroup of $G$, then $M=G/H$ is a closed $G$-manifold and $EG\times_GM\simeq BH$ is self dual with respect to Morava $K$-theory by (\ref{bgknd}). In general, suppose $G$ acts on a compact manifold $M$ smoothly. The action is said to be almost free if it has finite stabilizers. In such case, one can generalize the result of Ravenel for classifying spaces of finite groups mentioned above to show that $K(n)^{\ast}(EG\times_GM)$ has finite rank over $K(n)^{\ast}$. This finiteness condition is necessary for $EG\times_GM$ to satisfy a $K(n)$-version of Poincar\'e duality. Our main theorem concerns the $K(n)$-duality of such spaces. Given such an $G$-action on $M$ it gives rise to an orbifold with classifying space $EG\times_G M$. See \cite{Adem} for background on orbifolds. 
%Here we adopt a definition of orbifolds which is more general than the classical case and do not assume the $G$-action on $M$ to be effective. 
Thus our main theorem can be viewed as a $K(n)$-version of Poincar\'e duality for classifying spaces of quotient orbifolds. It is a conjecture that every orbifold is equivalent to a quotient orbifold arising from a smooth, almost free Lie group action on a smooth manifold. See \cite{Hen} for partial results towards this conjecture. These suggest that our result covers a large class of interesting orbifolds. 

We introduce a few terminologies for our main theorem. Let $G$ be a compact Lie group with adjoint representation $\mathfrak{g}$ and $M$ be a smooth $G$-manifold. Consider the $G$-vector bundles $E G \times T M \rightarrow E G \times M$ and $E G \times M \times \mathfrak{g} \rightarrow E G \times M$ induced by $T M
\rightarrow M$ and the projection. Their $G$-orbits are the non-equivariant
vector bundles $E G \times_G T M \rightarrow E G \times_G M$ and $(E G \times
M \times \mathfrak{g}) / G \rightarrow E G \times_G M$. For simplicity, we will
denote them by $T M$ and $\mathfrak{g}$ respectively. The virtual bundle $- T
M + \mathfrak{g}$ defines a Thom spectrum $(E G \times_G M)^{- T M
+\mathfrak{g}}$. Our main theorem below says that this Thom spectrum is the $K(n)$-dual of the suspension spectrum of $EG\times_GM$.

\begin{theorem}\label{mainquotient}
  Let $p$ be a prime and $K(n)$ be the corresponding $n$-th Morava $K$-theory. Let G
  be a compact Lie group acting on a smooth m-dimensional closed manifold $M$. Then
  the following holds:
  %\begin{innerlist}
  \begin{enumerate}[(a)]
      \item There is a map of spectra
    \begin{equation}
      \lambda_{G,M}:\Sigma^{\infty} (E G \times_G M)_+ \longrightarrow F ((E G \times_G
      M)^{- T M +\mathfrak{g}}, \mathbb{S}) \label{egmge}
    \end{equation}
    which reduces to (\ref{swd}) if $G$ is trivial and (\ref{bgd}) if $M$ is a
    point and $G$ is finite.    
    \item Suppose the $G$-action on $M$ is almost free.
    Then the map (\ref{egmge}) induces $K (n)$-duality
    \[ 
       K (n)_{\ast} (E G_{} \times_G M) \cong \widetilde{K (n)}^{- \ast} ((E G \times_G
       M)^{- T M +\mathfrak{g}}) %K (n)^{- \ast} ((E G \times_GM)^{- T M +\mathfrak{g}},pt) . 
    \]
    \item Suppose in addition to the condition in (b), $p > 2$, $T M$ is
    orientable and the $G$-actions on both $T M$ and $\mathfrak{g}$ are
    orientation preserving. Then
    \[
      K (n)_{\ast} (E G \times_G M) \cong K (n)^{m - \dim G - \ast} (E G
      \times_G M). 
    \]
  \end{enumerate}
  %\end{innerlist}
\end{theorem}

\begin{remark}
\begin{enumerate}[(i)]
\item
Theorem \ref{mainquotient}(c) is a consequence of \ref{mainquotient}(b) by the fact that $K(n)$-orientability is the same as ordinary orientability for $p>2$ \cite{Rudyak}.
\item 
There is a version of $K(n)$-duality for compact $G$-manifolds with boundary
\[
K (n)_{\ast} (EG\times_GM,EG\times_G{\partial M})\cong K(n)^{m-\dim G-\ast} (EG\times_GM).
\]
\end{enumerate}
\end{remark}

\subsection{\textit{K}(\textit{n}) self-duality for Deligne-Mumford stacks}
Our map (\ref{egmge}) is motivated by the similarities between the constructions of the duality maps (\ref{swd}) for $M$ and (\ref{bgd}) for $BG$. To make these similarities more transparent, we will need the notion of stack, which is a generalization of spaces. Both $M$ and $BG$ are the homotopy types of certain simple stacks, and this setup puts the two spaces on an equal footing. Indeed, the map (\ref{egmge}) can be regarded as the stack version of the Spanier-Whitehead map. To elaborate this point, we first recall some important facts about stacks. See \cite{H},\cite{NoohiTopSt} for a detailed introduction on the subject. 

Let $G$ be a compact Lie group and $M$ be a smooth $G$-manifold. They define a differentiable quotient stack $\mfX=[M/G]$. If the group action is almost free, then $\mfX$ is a Deligne-Mumford stack. A differentiable Deligne-Mumford stack is the same as an orbifold. The dimension $\dim\mfX$ of $\mfX$ is equal to $\dim M- \dim G$. Every differentiable stack $\mfX$ has an associated homotopy type $\Ho(\mfX)$, which can be defined in a functorial way {\cite{Noohi},\cite{EbertHo}. For quotient stacks, $\tmop{Ho}([M/G])$ is given by the Borel construction $EG\times_G M$. It is well-defined up to homotopy in the sense that an equivalence of stacks $[M/G]\simeq[N/H]$ induces a canonical weak equivalence 
\begin{equation}\label{weho}
EG\times_GM\simeq EH\times_HN. 
\end{equation}
%For a homology theory $E$, $E_n(\mfX)$ and $E^n(\mfX)$ are defined to be $E_n(\Ho(\mfX))$ and $E^n(\Ho(\mfX))$ respectively. 

There is also a notion of tangent stacks for differentiable stacks. If $\mfX=[M/G]$ is a Deligne-Mumford stack, its tangent stack $T\mfX$ defines a vector bundle over $\Ho(\mfX)=EG\times_G M$, stably equivalent to the virtual bundle $TM-\mathfrak{g}$ as in theorem \ref{mainquotient}. Hence, the Thom spectrum $(EG\times_GM)^{-TM+\mathfrak{g}}$ in (\ref{egmge}) represents $\Ho(\mfX)^{- T\mathfrak{X}}$. %$\mfX$ is orientable if $M$ is orientable and the $G$-action is orientation preserving. 
Similar to homotopy type, an equivalence of Deligne-Mumford stacks $[M/G]\simeq[N/H]$ induces a canonical weak equivalence 
\begin{equation}\label{wethom}
(EG\times_GM)^{-TM+\mathfrak{g}} \simeq (EH\times_HN)^{-TN+\mathfrak{h}}. 
\end{equation}
We will show that in this situation, $\lambda_{G,M}$ and $\lambda_{H,N}$ as in (\ref{egmge}) commute with the canonical equivalences (\ref{weho}) and (\ref{wethom}).

\begin{proposition}\label{prop:gmhn}
Suppose $[M/G]\simeq[N/H]$ is an equivalence between Deligne-Mumford stacks arising from compact Lie group actions on smooth closed manifolds. Then the maps $\lambda_{G,M},\lambda_{H,N}$ as in (\ref{egmge}) satisfy the following commutative diagram 
\begin{equation}\label{indeptmg}
\begin{aligned}
\xymatrix{
\Sigma^{\infty}(E G \times_G M)_+ \ar@{-}^\simeq[d]\ar^-{\lambda_{G,M}}[rr] && F ((E G \times_G
      M)^{- T M +\mathfrak{g}}, \mathbb{S})\ar@{-}^\simeq[d]\\
\Sigma^{\infty}(E H \times_H N)_+ \ar^-{\lambda_{H,N}}[rr] && F ((E H \times_H
      N)^{- T N +\mathfrak{h}}, \mathbb{S}),\\
}
\end{aligned}
\end{equation}
where the vertical weak equivalences are induced by (\ref{weho}) and (\ref{wethom}).
\end{proposition}

Note that if a differentiable Deligne-Mumford stack $\mfX$ admits an equivalence $\mfX\simeq[M/G]$, then $\Ho(\mfX)\simeq EG\times_GM$ and homotopically (\ref{egmge}) can be expressed as $\lambda_{\mfX}:\Sigma^{\infty} \tmop{Ho}(\mathfrak{X})_+ \to F (\tmop{Ho}(\mathfrak{X})^{- T\mathfrak{X}},\mathbb{S})$. The upshot of proposition \ref{prop:gmhn} is that $\lambda_{\mfX}$ is  a well-defined map of the stack $\mfX$, independent of the choices of $G$ and $M$, in the sense that the diagram (\ref{indeptmg}) commutes. As mentioned before, $\lambda_{\mfX}$ can be considered as the stack version of Spanier-Whitehead map. Analogous to the case of manifolds, it is obtained from the Pontryagin-Thom construction of the diagonal map $\mfX\to\mfX\times\mfX$. There is a general Pontryagin-Thom construction for local quotient stacks due to Ebert and Giansiracusa \cite{Ebert}.

In light of Proposition 1.3, the result in theorem \ref{mainquotient} can be interpreted as the following Poincar\'e duality for the Deligne-Mumford stack $\mfX = [M/G]$ in Morava $K$-theory.

\begin{corollary}\label{mainstack}

Let $p$ be a prime and $K(n)$ be the corresponding $n$-th Morava $K$-theory. Let $\mathfrak{X}$ be a Deligne-Mumford stack arising from an almost free action of a compact Lie group on a smooth closed manifold. Then the following holds:
  \begin{enumerate}[(a)]
      \item There is a map of spectra
    \begin{equation}
      \lambda_{\mfX}:\Sigma^{\infty} \tmop{Ho}(\mathfrak{X})_+ \longrightarrow F (\tmop{Ho}(\mathfrak{X})^{- T\mathfrak{X}},\mathbb{S}) \label{stackswd}
    \end{equation}
    %which reduces to (\ref{swd}) if $\mathfrak{X}$ is a differentiable manifold and (\ref{bgd}) if $\mathfrak{X}$ is the quotient stack $[\ast/G]$ for a finite group $G$. 
    which induces $K(n)$-duality
    \begin{equation} K (n)_{\ast} (\Ho(\mathfrak{X})) \cong \widetilde{K (n)}^{- \ast}(\Ho(\mathfrak{X})^{- T\mathfrak{X}}). \label{stackd} \end{equation}
    \item Suppose $p > 2$ and $\mathfrak{X}$ is oriented. Then
    \begin{equation}
      K (n)_{\ast} (\Ho(\mathfrak{X})) \cong K (n)^{\dim \mathfrak{X} - \ast} (\Ho(\mathfrak{X})) . \label{stackpd}
    \end{equation}
  \end{enumerate}

\end{corollary}

\subsection{Intersection theory on stacks}
As in the case of Poincar\'e duality for manifolds, two consequences of the $K(n)$-duality of Deligne-Mumford stacks are the definitions of fundamental class and intersection product in homology. 

\begin{definition}\label{fc}
Let $\mfX$ be an oriented $q$-dimensional Deligne-Mumford stack arising from the action of a compact Lie group on a smooth closed manifold.
Let $(\lambda_{\mfX})_{\ast}:K (n)_{\ast} (\Ho(\mathfrak{X})) \cong K (n)^{q - \ast}(\Ho(\mfX))$ be the isomorphism (\ref{stackpd}) if $p>2$ or (\ref{stackd}) if $q=0$.
\begin{enumerate}[(a)]
	\item The $K(n)$-fundamental class of $\mathfrak{X}$ is defined to be
  $[\mathfrak{X}]=(\lambda_{\mfX})_{\ast}^{-1}(1) \in K (n)_q ( \Ho(\mathfrak{X}))$, the pre-image of $1\in K (n)^0 (\Ho(\mfX))$ under $(\lambda_{\mfX})_{\ast}$.
  \item The intersection product $$\cap:K(n)_i(\Ho(\mfX))\times K(n)_j(\Ho(\mfX)) \to K(n)_{i+j-q}(\Ho(\mfX))$$ 
in $K(n)_{\ast}(\Ho(\mfX))$ is defined to be the dual of cup product with respect to $(\lambda_{\mfX})_{\ast}$. More precisely,
$\alpha\cap\beta:=(\lambda_{\mfX})_{\ast}^{-1}((\lambda_{\mfX})_{\ast}(\aaa)\cup(\lambda_{\mfX})_{\ast}(\bbb)).$
\end{enumerate}
\end{definition}

Since $(\lambda_{\mfX})_{\ast}([\mfX])=1$, the fundamental class $[\mfX]$ acts as identity in intersection product. A $K(n)^{\ast}$-valued integration homomorphism
  \begin{equation*}
  \int_{\mfX}:K(n)^{\ast}(\Ho(\mfX))\to K(n)^{\ast-q}(pt)
  \end{equation*}
can also be defined by evaluating cohomology classes on $[\mfX]$.  It generalizes the integration map $H^q(\Ho(\mfX);\mathbb{Q})\to\mathbb{Q}$ in the case $K(0)=H\mathbb{Q}$. Such a $\mathbb{Q}$-valued integration map is used in different areas of mathematics such as Gromov-Witten theory.

A key difference between manifolds and Deligne-Mumford stacks is that the later possess singularities of finite order. A Deligne-Mumford stack is pointwisely equivalent to $[pt/G]$ for some finite group $G$. Therefore, in order to understand intersection product, it is interesting to begin with $K(n)_{\ast}(\Ho([pt/G]))=K(n)_{\ast}(BG)$.

For manifolds, the intersection product in ordinary homology admits a geometric interpretation for transverse submanifolds. We look for an analogous statement for stacks defined by finite groups. Submanifolds can be replaced by subgroups. For the transversality condition, we propose the following definition.
\begin{definition}\label{def:transv}
Two subgroups $H$ and $K$ of a finite group $G$ is said to intersect transversely if $HK:=\{hk|h\in H,k\in K\}=G$.
\end{definition}

For any subgroup $i:H\hookrightarrow G$, write $[BH]$ for the image of the fundamental class of $[pt/H]$ under the map $i_{\ast}:K(n)_0(BH)\to K(n)_0(BG)$. The following is our intersection formula for transverse subgroups of $G$. 
 
\begin{theorem}\label{thm:transversecap}
Suppose the $H,K$ are transverse subgroups of a finite group $G$. Then $[BH]\cap[BK]=[B(H\cap K)]$. 
\end{theorem}

The organization of the paper is as follow. Section 2 contains background materials of equivariant stable homotopy theory from \cite{LMS}. In section 3, we recall the basics of Tate spectrum, describe the duality map in our main theorem and explain the relation between the two. The proof of the main theorem will be presented in section 4. In section 5, we explain the construction of the duality map in our main theorem from the viewpoint of stacks and compare it with that of the Spanier-Whitehead duality of manifolds. Some examples and calculations of classifying spaces of finite groups will be given in section 6.

Our work generalizes \cite{Greenlees} for classifying spaces of finite groups to certain Deligne-Mumford stacks; recently, Hopkins and Lurie have obtained a generalization in another direction, replacing $BG=K(G,1)$ by a space with finitely many non-zero homotopy groups, all of which are finite. 

\subsection*{Acknowledgements}
The content of this paper formed a major part of the author's Ph.D. thesis at Stanford University. He would like to thank his Ph.D. advisor S\o ren Galatius for his insightful ideas and patient guidance throughout this project. The author is grateful to John Greenlees for his helpful suggestions. The author would also like to thank Haynes Miller for hosting at Massachusetts Institute of Technology.

\section{Recollection from equivariant stable homotopy theory}\label{sec:ESHT}
The main tool for the proof of our theorems is equivariant stable homotopy theory. In this section we will recall some basic definitions and theorems from \cite{LMS}. 

Let $G$ be a compact Lie group. We say a
real $G$-inner product space $U$ is a $G$-universe if it is a direct sum of finite dimensional $G$-representations and contains
countably infinite copies of the trivial representation and any other irreducible representations it contains. $U$ is called complete if it contains all irreducible finite dimensional representations. For finite dimensional representations $V\subset W$ in $U$, write $W - V = V^{\perp} \cap W$. 

From now on $U$ will be assumed to be a complete $G$-universe. A $G$-spectrum $D$ (indexed over $U$) consists of a pointed $G$-space $D(V)$ for each finite dimensional representation $V$ in $U$ and a based $G$-map $\sigma_{VW} : \Sigma^{W - V} D (V) \rightarrow D (W)$ for each pair $V \subset W$ of those representations such that the adjoints
$\widetilde{\sigma_{VW}}: D (V) \stackrel{\cong}{\rightarrow} \Omega^{W- V} D (W)$ 
are homeomorphisms. Here the $G$-action on $\Sigma^{W - V} D (V)$ and $\Omega^{W- V} D (W)$ is given by diagonal action and conjugation respectively. The basepoint of each $D(V)$ is assumed to be $G$-fixed. The maps $\sigma_{VW}$ are called the structure maps of $D$ and are required to satisfy $\sigma_{VV}=Id_{D(V)}$ and a compatibility condition for each triple $V \subset W \subset Z$ of finite dimensional representations in $U$. 

A morphism $f : D \rightarrow E$ between $G$-spectra consists a
collection of based maps $\{f_V : D (V) \rightarrow E (V)\}_V$ which commute with the structure maps of $D$ and $E$. The category of $G$-spectra is denote by $G\mathcal{S} U$.

We will also work with $G$-spectra indexed on the smaller universe $U^N$, the
$N$-fixed points of $U$ for some normal subgroup $N\subset G$. The category of $G$-spectra indexed on $U^N$ is denoted by $G\mathcal{S}U^N$. Note that if $N=G$ then $U^G\cong \R^{\infty}$ is $G$-trivial. $G$-spectra indexed on this trivial universe $U^G$ are called naive $G$-spectra.   

Let $i : U^N \rightarrow U$ be the inclusion. Evidently there is a functor
$i^{\ast} : G \mathcal{S} U \rightarrow G \mathcal{S} U^N$ which forgets spaces
indexed on representations $V$ not contained in $U^N$. It has a left adjoint
$i_{\ast} : G \mathcal{S} U^N \rightarrow G \mathcal{S} U$. Also, there
is a functor $\varepsilon^{\ast} : J\mathcal{S} U^N \rightarrow
G\mathcal{S} U^N$, where $J=G/N$, which assigns $G$-action to $D\in J\mathcal{S} U^N$ via the quotient map $\varepsilon:G\to J$. This functor has both left adjoint and right adjoint, namely taking $N$-orbit and $N$-fixed point respectively.
\[   G\mathcal{S} U^N  \underset{\varepsilon^{\ast}}{\overset{N \text{-orbit }}{\rightleftarrows}} 
   J\mathcal{S} U^N
\hspace{1cm}
   J\mathcal{S} U^N  \underset{N \text{-fixed point}}{\overset{\varepsilon^{\ast}}{\rightleftarrows}} G
   \mathcal{S} U^N
   \overset{i_{\ast}}{\underset{i^{\ast}}{\rightleftarrows}} G \mathcal{S}
   U \]

Analogous to $G$-CW complexes, $G$-CW spectra are spectra which are built from cell spectra $(G/H)_+\wedge \SSS^n$, where $H\subset G$ is a subgroup and $n\in\mathbb{Z}$. The set of homotopy classes of maps from $D$ to $E$ is denoted by $[D,E]$. The following theorem relates homotopy classes of maps of spectra indexed on different universes.

\begin{theorem}\cite[II, Theorem 2.8]{LMS}\label{isochangeofu}
  Suppose $N$ is a normal subgroup of compact Lie group $G$. Let $U$ be a complete $G$-universe and $i:U^N\to U$ be the inclusion. Let $D,E \in G \mathcal{S}U^N$ with $D$ a $N$-free $G$-CW spectrum. Then $i_{\ast}$ induces a bijection
$$ [D,E]_{G \mathcal{S}U^N}\cong [i_{\ast}D,i_{\ast}E]_{G \mathcal{S}U}. $$

\end{theorem}

Two important ingredients of equivariant stable homotopy theory in the proof of our main theorem are the equivariant Spanier-Whitehead duality (Theorem \ref{thm:GSW}) and a generalized Adams' isomorphism (Theorem \ref{LMStransfer}). We will state these results and explain the construction of the maps involved. 

Suppose $G$ is a compact Lie group and $M$ is a closed $G$-manifold. Let $Q\to M$ be a $G$-vector bundle such that $Q\oplus TM$ is isomorphic to a trivial vector bundle $M\times V$ for some finite dimensional $G$-representation $V$. The Thom spectrum $M^{-TM}\in G\mathcal{S}U$ is defined to be $\Sigma^{-V}\Sigma^{\infty}M^Q$. The equivariant version of Spanier-Whitehead theorem states that the suspension spectrum $\Sigma^{\infty}M_+$ and the Thom spectrum $M^{-TM}$ are dual to each other in the category of $G$-spectra.
\begin{theorem}\cite[III, Theorem 5.1]{LMS}\label{thm:GSW}
Let $G$ be a compact Lie group and $M$ be a closed $G$-manifold. Then there is a weak equivalence
\begin{equation}\label{gsw}
   \Sigma^{\infty}M_+\simeq F(M^{-TM},\SSS).
\end{equation}
between $G$-spectra in the category $G\mathcal{S}U$.
\end{theorem}

The map (\ref{gsw}) can be constructed in the following way which is slightly different from the construction given in \cite{LMS}. Let $Q\to M$ be a $G$-vector bundle such that $Q\oplus TM\cong M\times V$ for some finite dimensional $G$-representation $V$ as above. Consider the composition  
\[
M\xrightarrow{\Delta}M\times M\xrightarrow{1\times s} M \times Q
\]
where $s:M\to Q$ is the zero section. The normal bundle of this composition is the trivial bundle $M\times V$. Pontryagin-Thom construction gives a based map
\[M_+\wedge M^Q\rightarrow (M\times V)_+=\SSS^V\wedge M_+.\]
Since $Q\oplus TM\cong M\times V$, taking infinite suspension followed by desuspension $\Sigma^{-V}$ gives a map 
\begin{equation}  
\mu:\Sigma^{\infty} M_+ \wedge M^{- T M}\rightarrow \Sigma^{\infty} M_+\label{Mdiagpt}
\end{equation}
of $G$-spectra. By further composing with the infinite suspension of the collapse map $M_+\to pt_+=\SSS^0$ and taking adjoint, we obtain the equivalence (\ref{gsw}).
 
Another important theorem for us is the following generalization of Adams' isomorphism.

\begin{theorem}\cite[II, Special case of Theorem 7.1]{LMS}\label{LMStransfer}
  Let $\mathfrak{g}$ be the adjoint representation of $G$, $\epsilon:G\to G/G\cong\{1\}$ be the quotient map and $i:U^G\to U$ be the inclusion. Suppose $D \in G \mathcal{S} U^G$ is a $G$-free spectrum. Then there is a transfer map
  \begin{equation}\label{tau}
     \tau:i_{\ast}\epsilon^{\ast}(D/G) \longrightarrow \Sigma^{-\mathfrak{g}}i_{\ast}D
  \end{equation}
  whose adjoint 
 
  \[ \tilde{\tau}:D / G \overset{\simeq}{\longrightarrow} (\Sigma^{-\mathfrak{g}}i_{\ast}D)^G \]
is a weak equivalence in $\mathcal{S} U^G$.

\end{theorem}

Let us also recall the construction of (\ref{tau}) described in the proof of \cite[II, Theorem 7.1]{LMS}. To better fit our notation, we replace the group $\Gamma=G\times_cG$, the semi-direct product defined by the conjugation action of the first factor on the second one, in their construction by the isomorphic group $G^2$ via the isomorphism $(g,n)\mapsto(gn,g)$. Under this identification, the normal subgroup $\Pi=1\times_cG\subset \Gamma$ and the map $\theta,\epsilon:\Gamma\to G$ in their proof becomes $G\times 1$ and the projection $\pi_1,\pi_2:G^2\to G$ respectively. Their $\Gamma$-space $N=G$ is identified with the $G^2$-space $G'$ which has action $(g_1,g_2)g'=g_1g'g_2^{-1}$. We will explain the construction of the map (\ref{tau}) in terms of $G^2$. 

Let $U'$ be a complete $G^2$-universe. Then $(U')^{G\times 1}$ can be regarded as a complete $G$-universe $U$ through the identification $G^2/(G\times 1)\cong G$ induced by the projection map $\pi_2$. Let $i:U^G\to U$ and $j:U=(U')^{G\times 1}\to U'$ be the inclusions. 

Let $\mathfrak{g}_2$ be the $G^2$-representation $\pi_2^{\ast}\mathfrak{g}$ and
\begin{equation}\label{embedG'}
\iota_{G'}:G'\hookrightarrow W
\end{equation}
be a $G^2$-embedding of $G'$ into a finite dimensional $G^2$-representation $W$. We can assume $W$ contains the representation $\mathfrak{g}_2$ so that $W'=W-\mathfrak{g}_2$ is also a $G^2$-representation. By identifying $\mathfrak{g}$ with the left invariant tangent vector fields on $G$, the tangent bundle $TG'$ is isomorphic to the trivial bundle $G'\times \mathfrak{g}_2$. The normal bundle of $\iota_{G'}$ is the trivial bundle $G'\times W'$. Applying Pontryagin-Thom construction to $\iota_{G'}$ gives $\SSS^{W}\to \Sigma^{\infty}G'_+\wedge\SSS^{W'}$, whose desuspension by $W$ is a map
\begin{equation}\label{t}
t:\SSS\to \Sigma^{\infty}G'_+\wedge\SSS^{-\mathfrak{g}_2}
\end{equation}
in $G^2\mathcal{S}U'$. $t$ is called the pre-transfer map.

Let $D_1=\pi_1^{\ast}D\in G^2\mathcal{S}U^G$. Since $D$ is $G$-free, the spectrum $i_{\ast}D_1\in G^2\mathcal{S}U$ is $(G\times 1)$-free. By theorem \ref{isochangeofu}, there is a bijection
\begin{equation}\label{changeofud}
[i_{\ast}D_1,i_{\ast}D_1\wedge \Sigma^{-\mathfrak{g}_2}G'_+]_{G^2\mathcal{S}U}\cong[j_{\ast}i_{\ast}D_1,j_{\ast}(i_{\ast}D_1\wedge \Sigma^{-\mathfrak{g}_2}G'_+)]_{G^2\mathcal{S}U'}.
\end{equation}
Taking smash product of $j_{\ast}i_{\ast}D_1$ with $(\ref{t})$ gives a map
$$
j_{\ast}i_{\ast}D_1\wedge t:j_{\ast}i_{\ast}D_1=j_{\ast}i_{\ast}D_1\wedge\SSS \to j_{\ast}i_{\ast}D_1\wedge \Sigma^{-\mathfrak{g}_2}G'_+.
$$
Using the bijection (\ref{changeofud}), this corresponds to a map
$$
i_{\ast}D_1\to i_{\ast}D_1\wedge \Sigma^{-\mathfrak{g}_2}G'_+. 
$$
The map $\tau$ in $(\ref{tau})$ is the $(G\times 1)$-orbit of it under the canonical identifications
%$$ \epsilon^{\sharp}(D/G) \cong (i_{\ast}D_1)/(G\times 1) \text{ and } 
%\Sigma^{-\mathfrak{g}}i_{\ast}D \cong (i_{\ast}D_1\wedge \Sigma^{-\mathfrak{g}_2}G'_+)/(G\times 1).
%$$
\begin{equation}\label{idtr1} i_{\ast}\epsilon^{\ast}(D/G) \cong (i_{\ast}D_1)/(G\times 1)\end{equation} and 
\begin{equation}\label{idtr2} \Sigma^{-\mathfrak{g}}i_{\ast}D \cong (i_{\ast}D_1\wedge \Sigma^{-\mathfrak{g}_2}G'_+)/(G\times 1).\end{equation}

\section{Tate spectrum and the duality map}

\subsection{Tate spectrum} 

As mentioned in the introduction, Greenlees and Sadofsky proved that the Morava $K$-theory of the classifying space of a finite group is self dual \cite{Greenlees}. A main ingredient in their proof is the contractibility of the Tate spectrum of $K(n)$. In this section, we first recall the definition of Tate spectrum and their results. Secondly we will define our duality map (\ref{egmge}) in theorem \ref{mainquotient} and relate it to another Tate spectrum. For a general reference on Tate spectrum, see \cite{Tate}.

Let $G$ be a compact Lie group and $U$ be a complete $G$-universe. We will primarily work in three different categories of spectra, namely $G\mathcal{S}U$, $G\mathcal{S}U^G$ and $\mathcal{S}U^G$. Consider the $n$-th Morava $K$-theory spectrum. It is a non-equivariant spectrum, but can also be regarded as a naive $G$-spectrum with trivial $G$-action. We use the same notation $K(n)$ to represent it in both $\mathcal{S}U^G$ and $G\mathcal{S}U^G$. The inclusion $i:U^G\to U$ induces the change of universe functor $i_{\ast}$ and $i_{\ast}K(n)\in G\mathcal{S}U$ is a $G$-spectrum. 
%the category of $G$-spectra, naive $G$-spectra and non-equivariant spectra respectively. 

Consider the cofiber sequence of pointed spaces
\begin{equation}\label{EGcofiber}
EG_+\to\SSS^0\to \widetilde{EG}.  
\end{equation}
The first map collapses $EG$ to the non-basepoint of $\SSS^0$. The cofiber $\widetilde{EG}$ is homotopy equivalent to the unreduced suspension of $EG$. Let $X\in G\mathcal{S}U$ be a $G$-spectrum. The collapse map $EG_+\to\SSS^0$ induces a map
\begin{equation}\label{Fcollapse}
X \simeq F(\SSS^0,X)\to F(EG_+,X).
\end{equation}

The smash product of (\ref{EGcofiber}) and (\ref{Fcollapse}) gives the following commutative diagram in $G\mathcal{S}U$
\begin{equation}\label{tatediagram}
\begin{aligned}
\xymatrix{
EG_+\wedge X \ar[d]\ar[r] & X\ar[d]\ar[r] &\widetilde{EG}\wedge X\ar[d]\\
EG_+\wedge F(EG_+,X) \ar[r]& F(EG_+,X) \ar[r] & \widetilde{EG}\wedge F(EG_+,X).
}
\end{aligned}
\end{equation}

Both rows of (\ref{tatediagram}) are cofiber sequences. The Tate spectrum of $X$ is defined to be the last term 
$t_G(X)=\widetilde{EG}\wedge F(EG_+,X)$. Since $EG$ is contractible, the left vertical map, being the smash product of $EG_+$ with the non-equivariant equivalence $G$-map (\ref{Fcollapse}), is a $G$-equivalence. Hence, $t_G(X)$ is the cofiber of the composite
\begin{equation}\label{coll}
EG_+\wedge X\to X \to F(EG_+,X).
\end{equation}
Let $\mathfrak{g}$ be the adjoint representation of $G$ and $M$ be a closed $G$-manifold. Define a map $\aaa_{G,M}$ 
$$\aaa_{G,M}:EG_+\wedge \Sigma^{-\mathfrak{g}}\Sigma^{\infty}M_+\wedge i_{\ast}K(n)\to F(EG_+,\Sigma^{-\mathfrak{g}}\Sigma^{\infty}M_+\wedge i_{\ast}K(n))$$
%$$\aaa_{G,M}:EG_+\wedge \Sigma^{\infty-\mathfrak{g}}M_+\wedge i_{\ast}K(n)\to F(EG_+,\Sigma^{-\mathfrak{g}}M_+\wedge i_{\ast}K(n))$$
%$$\aaa_{G,M}:EG_+\wedge \Sigma^{-\mathfrak{g}}M_+\wedge i_{\ast}K(n)\to F(EG_+,\Sigma^{-\mathfrak{g}}M_+\wedge i_{\ast}K(n))$$
by taking $X=\Sigma^{-\mathfrak{g}}\Sigma^{\infty}M_+\wedge i_{\ast}K(n)$ in (\ref{coll}). By the $G$-freeness of $EG$, theorem \ref{LMStransfer} and Spanier-Whitehead duality,
\[
\tilde{\tau}:(EG\times_G M)_+ \wedge K (n) \simeq (EG_+ \wedge \Sigma^{-\mathfrak{g}}\Sigma^{\infty}M_+\wedge i_{\ast} K (n))^G 
\]
and
\begin{align*}
 F (EG_+, \Sigma^{-\mathfrak{g}}\Sigma^{\infty}M_+\wedge i_{\ast} K (n))^G 
  &\simeq F (EG_+\wedge\Sigma^{\mathfrak{g}}M^{-TM},  i_{\ast}K(n))^G\\
  &\simeq F ((EG_+\wedge\Sigma^{\mathfrak{g}}M^{-TM})/G, K(n))\\
  &\simeq F ((EG\times_G M)^{-TM+\mathfrak{g}}, K(n))
\end{align*} 
are equivalences in $\mathcal{S}U^G$. By taking $G$-fixed point of $\alpha_{G,M}$, we get a map
\begin{equation}\label{alphagmg}
 \alpha_{G,M}^G:(EG\times_G M)_+ \wedge K (n) \to F ((EG\times_G M)^{-TM+\mathfrak{g}}, K(n)) 
\end{equation}
with cofiber $t_G(\Sigma^{-\mathfrak{g}}\Sigma^{\infty}M_+\wedge i_{\ast} K (n))^G$.

As we will show in proposition \ref{alphalambda}, (\ref{alphagmg}) is equal to the smash product of $K(n)$ with the map (\ref{egmge}) in theorem \ref{mainquotient}. Therefore, theorem \ref{mainquotient}(b) is equivalent to the contractibility of $t_G(\Sigma^{-\mathfrak{g}}\Sigma^{\infty}M_+\wedge i_{\ast} K (n))^G$. Making use of the complex orientability of Morava $K$-theory and the result of Ravenel that $K(n)^{\ast}(BG)$ has finite rank \cite{Ravenel}, Greenlees and Sadofsky proved the contractibility of this Tate spectrum in the case $M$ is a point and $G$ is finite.

\begin{theorem}\label{thmGreenleestate}(\cite{Greenlees})
For a finite group $G$, $t_G(i_{\ast}K(n))\simeq \ast$.
\end{theorem}

The contractibility of the cofiber $t_G(i_{\ast}K(n))^G$ implies that     
\begin{equation} 
  \alpha_{G,pt}^G:BG_+ \wedge K (n) \simeq F (BG_+, K (n))\label{BGKFBG}
\end{equation}
is an equivalence and hence gives the self $K(n)$-duality of $BG$ by taking homotopy groups.

%Warning: $\alpha_{G,pt}^G$ is not induced by the collapse map $BG_+\to\SSS^0$.

\subsection{A \textit{K}(\textit{n})-duality map for equivariant manifolds} \label{subsec:ourduality}
As mentioned in the introduction, Strickland described the construction of (\ref{BGKFBG}) in a slightly different way \cite{Strickland}. Let us recall his construction. Suppose $G$ is a finite group. Let $EG^2$ and $G'$, a diffeomorphic copy of $G$, be $G^2$-space with action $(g_1, g_2)(x_1,x_2)=(g_1x, g_2y)$ and $(g_1, g_2) g' = g_1 g' g^{- 1}_2$ respectively. The collapse map of $G'$ to a point induces a covering map $E G^2 \times_{G^2} G' \rightarrow E G^2 \times_{G^2} pt $. Since $E G^2 \times_{G\times 1}G'$ is contractible with a free $G^2/(G\times 1)$-action, $E G^2 \times_{G^2} G'$ is homotopically the classifying space $BG$. The covering map can be regarded as a homotopy version of the diagonal map $BG\to BG^2$.
Let $U'$ be a complete $G^2$-universe and $i':(U')^{G^2}\to U'$ be the inclusion. Pick a $G^2$-embedding $\iota_{G'}:G'\hookrightarrow W$ of $G'$ into a $G^2$-representation $W$. The normal bundle of $\iota_{G'}$ is $G'\times W$. Applying Pontryagin-Thom construction to $\iota_{G'}$ gives us a map $\SSS^W\to \Sigma^WG'_+$, whose desuspension by $W$ is a morphism $t:\SSS\to\Sigma^{\infty}G'_+$ in $G^2\mathcal{S}U'$. This morphism $t$ is the special case of (\ref{t}) for finite $G$. By the $G^2$-freeness of $EG^2$ and theorem \ref{isochangeofu}, there exists a morphism $\beta:\Sigma^{\infty}EG^2_+\to\Sigma^{\infty}EG^2_+\wedge G'_+$ in $G^2\mathcal{S}(U')^{G^2}$ such that $i'_{\ast}\beta=EG^2_+\wedge t$. The $G^2$-orbit $\beta/G^2$ is the transfer map $$\Sigma^{\infty} (B G \times B G)_+ \to \Sigma^{\infty}EG^2_+\wedge_{G^2} G'_+ \simeq \Sigma^{\infty} B G_+$$ of (\ref{transfer}) in the introduction. Composing this transfer map with the collapse map $\Sigma^{\infty} B G_+ \rightarrow \mathbb{S}$ and taking adjoint give us (\ref{bgd}). As pointed out by Strickland in \cite{Strickland}, the smash product of $K(n)$ with (\ref{bgd}) is the map $\alpha_{G,pt}^G$ of (\ref{BGKFBG}).
 
We now define the map (\ref{egmge}) in theorem \ref{mainquotient}. Its construction combines those of the equivariant Spanier-Whitehead duality  (\ref{gsw}) and $K(n)$-duality map (\ref{bgd}) for $BG$.

Suppose $G$ is a compact Lie group and $U'$ is a complete $G^2$-universe. Then $U=(U')^{G\times 1}$ is a complete $G$-universe. Denote by $i':(U')^{G^2}=U^G\to U'$ and $i:U^G\to U$ the inclusions of trivial universe. Let $G'$ be a copy of $G$ with $G^2$-action $(g_1,g_2)g'=g_1g'g_2^{-1}$. For any $G$-space or $G$-spectrum $X$, let $X_i,i=1,2$, be the $G^2$-space or $G^2$-spectrum $\pi_i^{\ast}X$, where $\pi_1,\pi_2:G^2\to G$ be the projection to the first and second factor respectively. 

Let $M$ be a closed $G$-manifold. Define a $G^2$-equivariant map
\begin{equation}\label{deltagm}
\Delta_{G,M}:G'\times M_2\to M_1\times M_2
\end{equation}
by $\Delta_{G,M}(g,x)=(gx,x)$. Pick
\begin{enumerate}[(a)]
\item a finite dimensional $G^2$-representation $W$;
\item a $G$-vector bundle $Q\to M$ such that $Q\oplus TM=M\times V$ for some $G$-representation $V$; and
\item a $G^2$-embedding 
\begin{equation}\label{iota}
 \iota_{G,M}:G'\times M_2\to M_1\times Q_2 \times W \times \mathfrak{g}_2
\end{equation}
over $\Delta_{G,M}$.
\end{enumerate}
For instance, if $\iota_{G'}:G'\to W$ is an $G^2$-embedding as in (\ref{embedG'}), we can take $\iota_{G,M}$ to be the map defined by $\iota_{G,M}(g',x)=(g'x,(x,0),\iota_{G'}(g'),0)$. The stable normal bundle of $\iota_{G,M}$ is  $ \Delta_{G,M}^{\ast}(TM_1\times Q_2\times W \times \mathfrak{g}_2)-(TG'\times TM_2)\cong(G'\times M_2)\times V_2\times W$. Applying Pontryagin-Thom construction to $\iota_{G,M}$ followed by $\Sigma^{-V_2- W}\Sigma^{\infty}$ gives a morphism 
\[
(\Delta_{G,M})^!:\Sigma^{\infty}{M_1}_+\wedge M^{-TM+\mathfrak{g}}_2\to\Sigma^{\infty}(G'\times M_2)_+
\]
in $G^2\mathcal{S}U'$. Since the domain of the smash product $EG^2_+\wedge (\Delta_{G,M})^!$ is $G^2$-free, by theorem \ref{isochangeofu}, $EG^2_+\wedge (\Delta_{G,M})^!$ is the image of a morphism 
\begin{equation}\label{beta}
\begin{split}
\beta_{G,M}:&\Sigma^{\infty}(EG_1\times M_1)_+\wedge (EG_2\times M_2)^{-TM_2+\mathfrak{g}}\\
           &\to \Sigma^{\infty}(EG^2\times(G'\times M_2))_+
\end{split}
\end{equation}
in $G^2\mathcal{S}(U')^{G^2}$ under the change of universe functor $i'_{\ast}:G^2\mathcal{S}(U')^{G^2}\to G^2\mathcal{S}U'$. On passage to the $G^2$-orbit of $\beta_{G,M}$ we obtain
\begin{equation}\label{betaovergg}
\begin{split}
\beta_{G,M}/G^2:&\Sigma^{\infty}(EG\times_G M)_+\wedge (EG\times_G M)^{-TM+\mathfrak{g}}\\
&\to \Sigma^{\infty}(EG^2\times_{G^2}(G'\times M_2))_+
\end{split}
\end{equation}
in the category $\mathcal{S}(U')^{G^2}$ of non-equivariant spectra. Composing (\ref{betaovergg}) with the collapse map \begin{equation}\label{collapse}
c_{G,M}:\Sigma^{\infty}(EG^2\times_G(G'\times M_2))_+\to\Sigma^{\infty}pt_+=\SSS
\end{equation} and taking adjoint, we finally get the map
\begin{equation*}
\lambda_{G,M}:\Sigma^{\infty}(EG\times_G M)_+\to F((EG\times_G M)^{-TM+\mathfrak{g}},\SSS),
\end{equation*}
which is defined to be the map (\ref{egmge}) in theorem \ref{mainquotient}.

The following proposition relates the map $\alpha_{G,M}$ of (\ref{alphagmg}) with our duality map $\lambda_{G,M}$ of (\ref{egmge}). The special case $M$ is a point and $G$ is finite was observed by Strickland \cite{Strickland}. 
\begin{proposition}\label{alphalambda}
$\alpha_{G,M}^G=\lambda_{G,M}\wedge K(n)$.
\end{proposition}
\begin{proof}
By the definition of $\alpha_{G,M}$, the adjoint 
\[
\widetilde{\alpha_{G,M}^G}:(EG\times_G M)_+ \wedge K(n) \wedge (EG\times_G M)^{-TM+\mathfrak{g}}\to K(n),
\]
of $\alpha_{G,M}^G$ is the $G$-orbit of a naive $G$-spectrum morphism 
\[
\epsilon^{\ast}(EG\times_G M)_+ \wedge K(n) \wedge (EG\times M)^{-TM+\mathfrak{g}}\to K(n),
\]
whose image under the change of universe functor $i_{\ast}:G\mathcal{S}U^G\to G\mathcal{S}U$ is the composite
\begin{align}
&i_{\ast}\epsilon^{\ast}((EG\times_G M)_+ \wedge K(n)) \wedge EG_+\wedge M^{-TM}\wedge\SSS^{\mathfrak{g}}\nonumber\\
\xrightarrow[\hspace{1cm}]{\tau\wedge 1}{}&\Sigma^{-\mathfrak{g}}i_{\ast}(EG_+\wedge M_+ \wedge K(n))\wedge EG_+\wedge M^{-TM}\wedge\SSS^{\mathfrak{g}}\nonumber\\
\xrightarrow[\hspace{1cm}]{\cong}{}& EG_+\wedge \Sigma^{\infty} M_+ \wedge M^{-TM}\wedge EG_+\wedge i_{\ast}K(n)\nonumber\\
\xrightarrow[\hspace{1cm}]{1\wedge\mu\wedge 1}{}&EG_+ \wedge \Sigma^{\infty} M_+ \wedge EG_+\wedge i_{\ast}K(n)\nonumber\\
\xrightarrow[\hspace{1cm}]{c\wedge 1}{}& i_{\ast}K(n).\label{unravel1}
\end{align}
Here $\tau$ is the map (\ref{tau}) for the case $D=EG_+\wedge M_+ \wedge K(n)$, $\mu$ is the map (\ref{Mdiagpt}) and $c$ is the collapse map to a point. Recall from the construction of (\ref{tau}) that to define $\tau$, we identify
\[
i_{\ast}\epsilon^{\ast}((EG\times_G M)_+ \wedge K(n))\cong (i_{\ast}({EG_1}_+\wedge {M_1}_+\wedge K(n)))/(G\times 1),
\]
\[
\Sigma^{-\mathfrak{g}}i_{\ast}(EG_+\wedge M_+ \wedge K(n))\cong (i_{\ast}({EG_1}_+\wedge {M_1}_+\wedge K(n))\wedge\Sigma^{-\mathfrak{g}_2}G'_+)/(G\times 1)
\]
as in (\ref{idtr1}) and (\ref{idtr2}). Similarly, we have the following identifications
\[
EG_+\wedge \Sigma^{\infty} M_+ \wedge M^{-TM} \cong ({EG_1}_+\wedge \Sigma^{\infty} {M_1}_+\wedge G'_+\wedge M^{-TM}_2)/(G\times 1),
\]
\[
EG_+\wedge \Sigma^{\infty} M_+ \cong ({EG_1}_+\wedge G'_+\wedge \Sigma^{\infty} M_{2+})/(G\times 1)
\]
of spectra in $G\mathcal{S}U=G^2/(G\times 1)\mathcal{S}U'^{G\times 1}$.

The map $\mu:\Sigma^{\infty} M_+ \wedge M^{-TM}\to \Sigma^{\infty} M_+$ is defined using the Pontryagin Thom construction of the diagonal map $\Delta_M:M\to M\times M$. Let $$\Delta_M':G'\times M_2\to M_1\times G'\times M_2$$ be the $G^2$-map given by $\Delta_M'(g',x)=(g'x,g',x)$. Note that $EG_1\times \Delta_M'$ is a map between $G^2$-free spaces with $(G\times 1)$-orbit $(EG_1\times \Delta_M')/(G\times 1)=EG\times \Delta_M.$ Hence, Pontryagin-Thom construction of $\Delta_M'$ and desuspension induce a morphism
$$\mu':\Sigma^{\infty} {M_1}_+ \wedge G'_+ \wedge M^{-TM}_2\to G'_+\wedge\Sigma^{\infty}{M_2}_+$$ 
in $G^2\mathcal{S}U$ such that $({EG_1}_+\wedge\mu')/(G\times 1)=EG_+\wedge\mu$. This, together with the construction of $\tau$ given in section \ref{sec:ESHT}, imply that the composite (\ref{unravel1}) is the $(G\times 1)$-orbit of a morphism 
\[
i_{\ast}({EG_1}_+\wedge {M_1}_+\wedge K(n)) \wedge {EG_2}_+\wedge M^{-TM}_2\wedge\SSS^{\mathfrak{g}_2}\to i_{\ast}K(n)
\]
in $G^2\mathcal{S}U$, whose image under $j_{\ast}:G^2\mathcal{S}U\to G^2\mathcal{S}U'$ is the composite
\begin{align}
&j_{\ast}i_{\ast}({EG_1}_+\wedge {M_1}_+\wedge K(n)) \wedge {EG_2}_+\wedge M^{-TM}_2\wedge\SSS^{\mathfrak{g}_2}\nonumber\\
\xrightarrow[\hspace{1.3cm}]{1\wedge t \wedge 1}{}&j_{\ast}i_{\ast}({EG_1}_+\wedge {M_1}_+\wedge K(n))\wedge\Sigma^{-\mathfrak{g}_2}G'_+ \wedge {EG_2}_+\wedge M^{-TM}_2\wedge\SSS^{\mathfrak{g}_2}\nonumber\\
\xrightarrow[\hspace{1.3cm}]{\cong}{}& {EG_1}_+\wedge \Sigma^{\infty} {M_1}_+ \wedge G'_+ \wedge M^{-TM}_2\wedge {EG_2}_+ \wedge j_{\ast}i_{\ast}K(n)\nonumber\\
\xrightarrow[\hspace{1.3cm}]{1\wedge j_{\ast}\mu' \wedge 1}{}&{EG_1}_+ \wedge G'_+ \wedge \Sigma^{\infty} {M_2}_+ \wedge {EG_2}_+ \wedge j_{\ast}i_{\ast}K(n)\nonumber\\ 
\xrightarrow[\hspace{1.3cm}]{c\wedge 1}{}& j_{\ast}i_{\ast}K(n).\label{unravel2}
\end{align}
Here $t$ is the pre-transfer map (\ref{t}). It is clear that the composite $(1\wedge j_{\ast}\mu' \wedge 1)\circ(1\wedge t \wedge 1)$ of the first three maps in (\ref{unravel2}) is the smash product of ${EG_1}_+\wedge{EG_2}_+\wedge j_{\ast}i_{\ast}K(n)$ with a morphism $\Sigma^{\infty}{M_1}_+\wedge M_2^{-TM}\wedge\SSS^{\mathfrak{g}}\to\Sigma^{\infty}(G'\times M_2)_+$ obtained by applying Pontryagin-Thom construction and then desuspension $\Sigma^{-V_2-W}$ to the composite embedding

\[G'\times M_2\xrightarrow{\Delta'} M_1\times G'\times M_2 \xrightarrow{1 \times \iota_{G'}\times s} M_1\times W \times Q_2\times\mathfrak{g}_2.\] 

The factor $\iota_{G'}$ and $s:M_2\to Q_2\times \mathfrak{g}_2$ in the last map is (\ref{embedG'}) and the zero section respectively. Note that this composite embedding is the map $\iota_{G,M}$ of (\ref{iota}) defined in the construction of $\lambda_{G,M}$. Hence, the composite $(1\wedge \mu' \wedge 1)\circ(1\wedge t \wedge 1)$ in (\ref{unravel2}) is equal to $$EG^2_+\wedge(\Delta_{G,M})^!\wedge j_{\ast}i_{\ast}K(n)=j_{\ast}i_{\ast}\beta_{G,M}\wedge j_{\ast}i_{\ast}K(n)$$ 
by the definition of $\beta_{G,M}$ in (\ref{beta}). By theorem \ref{isochangeofu}, it follows that the composite $(1\wedge \mu \wedge 1)\circ(\tau \wedge 1)$ of the first three maps in (\ref{unravel1}) is $(i_{\ast}\beta_{G,M}/(G\times 1))\wedge i_{\ast}K(n)$ and $\widetilde{\alpha_{G,M}^G}=(c_{G,M}\circ(\beta_{G,M}/G^2))\wedge K(n)$, where $c_{G,M}$ is the collapse map (\ref{collapse}). This shows $\alpha_{G,M}^G=\lambda_{G,M}\wedge K(n)$. 
\end{proof}

\section{Proof of main theorem}

In this section we will prove theorem \ref{mainquotient} about the $K(n)$-duality of $G$-manifolds. We will first show it for $G$-manifolds of the form $G/H$, $H$ being a finite subgroup of $G$. The main idea of the proof in these special cases is to reduce the $K(n)$-duality of $EG\times_G (G/H)$ to that of $BH$ by relating the map $\lambda_{G,G/H}$ with $\lambda_{H,pt}$, which is known to be a $K(n)$-equivalence. The $K(n)$-duality for more general $G$-manifolds follows from the special cases by induction on cells. 
  
Let $G$ be a compact Lie group. To simplify notations, we will write $X_{hG}$ to denote the homotopy orbit $EG\times_GX$ of a $G$-space $X$. Consider a $G$-manifold $G/H$, where $H$ is a finite subgroup of $G$. Note that $G/H$ has trivial tangent bundle $G/H\times \mathfrak{g}$. Hence, the virtual bundle $\mathfrak{g}-TM$ over $M$ is zero and we have
%$$\alpha_{G,G/H}^G:\Sigma^{\infty}(EG\times_G (G/H))_+\wedge K(n) \to F((EG\times_G (G/H))_+,K(n)).$$ 
$$\alpha_{G,G/H}^G:\Sigma^{\infty}{(G/H)_{hG}}_+\wedge K(n) \to F({(G/H)_{hG}}_+,K(n)).$$ 
Homotopically $(G/H)_{hG}=EG\times_G (G/H)$ is the classifying space $BH$. Hence, to show that $\alpha_{G,G/H}^G$ induces the $K(n)$-duality of $(G/H)_{hG}$, we want to relate it to $\alpha_{H,pt}^H$, which is a weak equivalence by theorem \ref{thmGreenleestate}. By proposition \ref{alphalambda}, it suffices to compare $\lambda_{G,G/H}$ and $\lambda_{H,pt}$. To do this, we consider $G$ as a $(G\times H)$-manifold with action $(g,h)x=gxh^{-1}$. Then $G/(G\times 1)\cong pt$ as $H$-spaces and $G/(1\times H)\cong G/H$ as $G$-spaces. It allows us to relate the two maps using $\lambda_{G\times H,G}$. 

%Consider the map $\Delta_{G,M}:G'\times M_2 \to M_1 \times M_2$ as defined by (\ref{deltagm}). Recall that $G'$ is a copy of $G$ with $G^2$-action $(g_1,g_2)g'=g_1g'g_2^{-1}$ and for any $G$-space or $G$-spectrum $X$, we denote by $X_i,i=1,2$, the $G^2$-space or $G^2$-spectrum $\pi_i^{\ast}X$ where $\pi_1,\pi_2:G^2\to G$ are the projections. 

Consider $\Delta_{G\times H,G}$ and $\Delta_{G,G/H}$ as defined by (\ref{deltagm}) and the maps $\pi_{G'}\times \pi_{G/H}:G'\times H' \times G_2 \to G' \times (G/H)_2$ between their domains and $\pi_{G/H}^2:G_1\times G_2 \to (G/H)_1\times(G/H)_2$ between their codomains. Both $\pi_{G'}\times \pi_{G/H}$ and $\pi_{G/H}^2$ are equivariant with respect to $(G\times H)^2\to G^2$ and induce diffeomorphisms
\begin{equation}\label{GHGGGH}
(G'\times H'\times G_2)/(1\times H)^2 \cong G'\times (G/H)_2
\end{equation}
and
\begin{equation}\label{GGGHGH}
(G_1\times G_2)/(1\times H)^2 \cong  (G/H)_1\times (G/H)_2. 
\end{equation}
The two $((G\times H)^2/(1\times H))$-spaces can thus be identified with the corresponding $G^2$-spaces. Under these identifications, $\Delta_{G\times H,G}/(1\times H)^2$ is $\Delta_{G,G/H}$ and so $(EG^2\times\Delta_{G\times H,G})/(G\times H)^2$ becomes $EG^2\times_{G^2}\Delta_{G,G/H}=(\Delta_{G,G/H})_{hG^2}$. 

Similarly, The maps $\pi_{H'}:G'\times H' \times G_2 \to H'$ and $\pi_{pt}^2:G_1\times G_2 \to pt$ are equivariant with respect to $(G\times H)^2\to H^2$ and enable us to identify $\Delta_{G\times H,G}/(G\times 1)^2$ with $\Delta_{H,pt}$ and $(EH^2\times\Delta_{G\times H,G})/(G\times H)^2$ with $EH^2\times_{H^2}\Delta_{H,pt}=(\Delta_{H,pt})_{hH^2}$. 

Under these identifications, the equivariant maps discussed above induce the following commutative diagram
\begin{equation}\label{3coverings}
\begin{aligned}
\xymatrix{
(G'\times (G/H)_2)_{hG^2}\ar^-{\Delta_{G,G/H}}[rrr] &&& ((G/H)_{hG})^2\\
(G'\times H'\times G_2)_{h(G\times H)^2}\ar^-{\Delta_{G\times H,G}}[rrr] \ar_{\pi_{G'}\times\pi_{G/H}}[u]\ar^{\pi_{H'}}[d]&&& (G_{h(G\times H)})^2\ar_{\pi_{G/H}^2}[u]\ar^{\pi^2_{pt}}[d]\\
H'_{hH^2}\ar^-{\Delta_{H,pt}}[rrr]&&& pt_{hH^2}\\
}
\end{aligned}
\end{equation}
on homotopy orbits. Here, we use the same symbol to represent an equivariant map and its induced map the homotopy orbits.
Note that all the vertical maps are weak equivalences since they can be expressed as the $(G\times H)^2$-orbits of non-equivariant weak equivalences between $(G\times H)^2$-free spaces. For instance, $\pi_{G'}\times\pi_{G/H}$ is the map
\begin{align*}
&(EG^2\times EH^2)\times_{(G\times H)^2} (G'\times H'\times G_2)\\
\to{}& EG^2\times_{(G\times H)^2} (G'\times H'\times G_2)\\
={}& EG^2\times_{G^2} ((G'\times H'\times G_2)/H^2)\\
={}& EG^2\times_{G^2} (G'\times (G/H)_2)
\end{align*}
induced by the projection $EG^2\times EH^2\to EG^2$. Also, the three horizontal maps of diagram (\ref{3coverings}) are finite covering maps. The middle one can be regarded as the pullback of the top one along $\pi_{G/H}^2$ or the bottom one along $\pi_{pt}^2$. The maps $\beta_{G,G/H}/G^2,\beta_{G\times H,G}/(G\times H)^2$ and $\beta_{H,pt}/H^2$, which are stable maps in the reversed directions as defined by (\ref{beta}) and (\ref{betaovergg}), can be considered as the associated transfer maps of these three covering maps. They satisfy the following Mackey property.

\begin{proposition} \label{MackeyGHG}
The following diagram

\begin{equation}\label{3PT}
\begin{aligned}
\xymatrix{
\Sigma^{\infty}((G/H)_{hG})^2_+\ar^-{\beta_{G,G/H}/G^2}[rrr] &&& \Sigma^{\infty}{(G'\times (G/H)_2)_{hG^2}}_+\\
\Sigma^{\infty}(G_{h(G\times H)})^2_+\ar^-{\beta_{G\times H,G}/(G\times H)^2}[rrr] \ar_{\simeq}[u]\ar^{\simeq}[d]&&& \Sigma^{\infty}{(G'\times H'\times G)_{h(G\times H)^2}}_+\ar_{\simeq}[u]\ar^{\simeq}[d]\\
\Sigma^{\infty}(pt_{hH})^2_+\ar^-{\beta_{H,pt}/H^2}[rrr]&&& \Sigma^{\infty}{(H'\times pt)_{hH^2}}_+\\
}
\end{aligned}
\end{equation}
commutes.
\end{proposition}
\begin{proof}
To define the maps in the diagram, we have to work in a complete $(G\times H)^2$-universe $\overline{U}$. The fixed point sets $U'=\overline{U}^{1\times H^2}$ is a complete $G^2$-universe and $U''=\overline{U}^{G^2\times 1}$ is a complete $H^2$-universe. Let $i',i'',\overline{i}$ be the inclusion of the trivial universe $U_0:=(U')^{G^2}=(U'')^{H^2}=\overline{U}^{(G\times H)^2}$ to $U',U''$ and $\overline{U}$ respectively. Also, let $j':U'\to\overline{U}$ and $j'':U''\to\overline{U}$ be inclusions to $\overline{U}$.

Recall that to define $\beta_{G\times H,G}$, one choose a $(G\times H)^2$-embedding $\iota_{G\times H,G}$ over $\Delta_{G\times H, G}$ as in (\ref{iota}). Since $TG=G\times \mathfrak{g}$, we can pick $Q=G\to G$ to be the zero bundle and $V=\mathfrak{g}$. Hence,  
$$\iota_{G\times H,G}:G'\times H'\times G_2\to G_1\times G_2 \times W \times \mathfrak{g}_2$$ for some $(G\times H)^2$-representation $W$.
Applying Pontryagin Thom construction to $\iota_{G\times H,G}$ followed by desuspension $\Sigma^{-W-\mathfrak{g}_2}$ yields a morphism 
$$\Delta_{G\times H,G}^!:\overline{i}_{\ast}\Sigma^{\infty} (G_1\times G_2)_+ \to \overline{i}_{\ast}\Sigma^{\infty} (G'\times H'\times G)_+$$ 
in $(G\times H)^2\mathcal{S}\overline{U}$. The map 
$$\beta_{G\times H,G}:\Sigma^{\infty}(EG^2\times EH^2\times G_1 \times G_2)_+\to \Sigma^{\infty}(EG^2\times EH^2\times G' \times H' \times G_2)_+$$
is defined to be a morphism in $(G\times H)^2\mathcal{S}U_0$ such that $\overline{i}_{\ast}\beta_{G\times H,G}=EG_+^2\wedge EH_+^2\wedge\Delta_{G\times H,G}^!$. Its existence and uniqueness up to homotopy are guaranteed by theorem \ref{isochangeofu} due to the $(G\times H)^2$-freeness of its domain. Note that $\Delta^!_{G\times H,G}$, and hence $\beta_{G\times H,G}$, are well-defined, independent of the choices of $W$ and $\iota_{G\times H,G}$. We will make use of this fact and pick different $W$ and $\iota_{G\times H,G}$ to relate $\beta_{G\times H,G}$ with $\beta_{G,G/H}$ and $\beta_{H,pt}$.

Let $W_G$ be a $G^2$-representation and $\iota_G:G'\to W_G$ be a $G^2$-embedding of $G'$. $W_G$ can be regarded as a $(G\times H)^2$-representation through the projection $(G\times H)^2\to G^2$. To define $\Delta^!_{G\times H,G}$, one can take $W=W_G$ and $\iota_{G\times H,G}=\iota_1$, where $\iota_1(g',h',x)=(g'x(h')^{-1},x,\iota_G(g'),0)$. Since $W_G$ and $\mathfrak{g}_2$ are fixed by $1\times H^2$, They are in the universe $U'$. Thus, applying Pontryagin-Thom construction to $\iota_1$ followed by desuspension $\Sigma^{-W_G-\mathfrak{g}_2}$ defines a morphism 
$$\Delta^!_1:i_{\ast}\Sigma^{\infty} (G_1\times G_2)_+ \to i_{\ast}\Sigma^{\infty} (G'\times H'\times G)_+$$ 
in $(G\times H)^2\mathcal{S}U'$ with $j'_{\ast}\Delta^!_1=\Delta_{G\times H, G}^!$. 

By theorem \ref{isochangeofu} and the $(G\times H)^2$-freeness of $EG^2\times G_1\times G_2$, there exists a unique homotopy class of morphisms 
$$\beta_1:\Sigma^{\infty}(EG^2\times G_1\times G_2)_+\to \Sigma^{\infty}(EG^2\times G'\times H' \times G_2)_+$$ 
in $(G\times H)^2\mathcal{S}U_0$ such that $i'_{\ast}\beta_1=EG^2_+\wedge \Delta^!_1$. Since 
\begin{align*}
\overline{i}_{\ast}(EH_+^2\wedge \beta_1)&=EH_+^2\wedge j'_{\ast}i'_{\ast}\beta_1\\
&=EH_+^2\wedge j'_{\ast}( EG^2_+\wedge \Delta^!_1)\\
&=EG^2_+\wedge EH^2_+ \wedge \Delta_{G\times H,G}^!, 
\end{align*}
we have $EH_+^2\wedge \beta_1=\beta_{G\times H,G}$.

Also, note that under the identifications (\ref{GHGGGH}) and (\ref{GGGHGH}), $\iota_1/(1\times H)^2$ is a $G^2$-embedding $G'\times (G/H)_2\to (G/H)_1\times (G/H)_2\times W_G \times \mathfrak{g}_2$ over $\Delta_{G,G/H}$. Hence, $\Delta^!_1/(1\times H)^2=\Delta_{G,G/H}^!$. It follows that 
$$i'_{\ast}(\beta_1/(1\times H)^2)=(EG^2_+\wedge \Delta^!_1)/(1\times H)^2=EG^2_+\wedge\Delta_{G,G/H}^!$$ and so by theorem \ref{isochangeofu}, $\beta_1/(1\times H)^2=\beta_{G,G/H}.$ 

The results above show that $\beta_{G\times H,G}/(G\times H)^2=(EH^2_+\wedge \beta_1)/(G\times H)^2$ and $\beta_{G,G/H}/G^2=\beta_1/(G\times H)^2$. This proves the commutativity of the upper half of the diagram in the proposition.

Similarly, suppose $\iota_H:H'\to W_H$ is a $H^2$-embedding of $H'$ into a $H^2$-representation $W_H$, which can be regarded as a $(G\times H)^2$-representation through $(G\times H)^2\to H^2$. We take $W=W_H$ and $\iota_{G\times H,G}=(\iota_2,0)$, where 
$$\iota_2:G'\times H'\times G_2\to G_1\times G_2 \times W_H$$ 
is given by $\iota_2(g',h',x)=(g'x(h')^{-1},x,\iota_H(h'))$. By applying Pontryagin-Thom construction to $\iota_2$ followed by desuspension $\Sigma^{-W_H}$, we can define morphisms $\Delta^!_2$ in $(G\times H)^2\mathcal{S}U''$ and $\beta_2$ in $(G\times H)^2\mathcal{S}U_0$ such that $j''_{\ast}\Delta^!_2=\Delta_{G\times H, G}^!$, $\Delta^!_2/(G\times 1)^2=\Delta_{H,pt}^!$ and $i''_{\ast}\beta_2=EH^2_+\wedge\Delta^!_2$. From these we can deduce that $\beta_{G\times H,G}/(G\times H)^2=(EG^2_+\wedge \beta_2)/(G\times H)^2$ and $\beta_{H,pt}/H^2=\beta_2/(G\times H)^2$. The commutativity of the lower half of the diagram in the proposition follows.
\end{proof}

Using proposition \ref{MackeyGHG}, we prove the following result about Tate spectrum.

\begin{proposition}\label{tatefixcollapse}
Let $G$ be a compact Lie group and $Y$ be a $G$-CW spectrum built from cells of the form $(G/H)_+\wedge \SSS^k$, where $H$ is a finite subgroup of $G$. Then the $G$-fixed point spectrum $t_G(\Sigma^{-\mathfrak{g}}Y\wedge i_{\ast} K (n))^G$ is contractible.
\end{proposition}

\begin{proof}
Clearly $t_G(\Sigma^{-\mathfrak{g}}Y\wedge i_{\ast} K (n))^G$ is contractible when $Y$ is a point. By induction, it suffices to prove the statement for the special case $Y=\Sigma^{\infty} (G/H)_+$, where $H$ is a finite subgroup of $G$. Since $t_G(\Sigma^{-\mathfrak{g}}G/H_+\wedge i_{\ast} K (n))^G$ is the cofiber of $\alpha_{G,G/H}^G$, the special case can be proved by showing $\alpha_{G,G/H}^G$ is an equivalence. 

Consider diagram (\ref{3PT}) of proposition \ref{MackeyGHG}. By post-composing the three horizontal maps with the corresponding collapse maps to a point and taking adjoint, we obtain the commutative diagram
\begin{equation}\label{lambdaGHG}
\begin{aligned}
\xymatrix{
\Sigma^{\infty}{(G/H)_{hG}}_+\ar^{\lambda_{G,G/H}}[rr] && F({(G/H)_{hG}}_+,\SSS)\ar^{\simeq}[d]\\
\Sigma^{\infty}{G_{h(G\times H)}}_+ \ar^-{\lambda_{G\times H,G}}[rr] \ar_{\simeq}[u]\ar^{\simeq}[d]&& F({G_{h(G\times H)}}_+,\SSS)\\
\Sigma^{\infty}{pt_{hH}}_+\ar^{\lambda_{H,pt}}[rr]&& F({pt_{hH}}_+,\SSS).\ar_{\simeq}[u]\\
}
\end{aligned}
\end{equation}
By proposition \ref{alphalambda}, the smash product of $K(n)$ with the top horizontal map and the bottom horizontal map is $\alpha_{G,G/H}^G$ and $\alpha_{H,pt}^H$ respectively. Since $\alpha_{H,pt}^H$ is a weak equivalence, so is $\alpha^G_{G,G/H}$ by the commutativity of the diagram. This completes the proof.
\end{proof}

\begin{remark} 
Diagrams (\ref{3PT}) and (\ref{lambdaGHG}) are special cases of (\ref{PTmpn}) and (\ref{lambdampn}) in section \ref{sec:stack}. 
From the point of view of stacks, the equivariant $G$-space $G/H$, $H$-space $pt$ and $(G\times H)$-space $G$ represent the same differentiable stack $\mfX$. In each of the diagrams (\ref{3PT}) and (\ref{lambdaGHG}), the three horizontal maps represent the same map of $\mfX$ homotopically. We will explain these in more details in section \ref{sec:stack}.
\end{remark}

Our main theorem is an easy consequence of proposition \ref{tatefixcollapse}.

\begin{proof}[Proof of theorem \ref{mainquotient}]
Part (a) is clear from the construction of $\lambda_{G,M}$ in section \ref{subsec:ourduality}. If the stabilizer subgroup $G_x$ is finite for all $x \in M$, then $M$ is a $G$-CW complex which can be built from cells of the form $D^n\times G/H$ with $H$ a finite subgroup of $G$. Hence, $t_G(\Sigma^{-\mathfrak{g}}M_+\wedge i_{\ast} K (n))^G$ is contractible by proposition \ref{tatefixcollapse}. It follows that $\alpha_{G,M}^G$, whose cofiber is $t_G(\Sigma^{-\mathfrak{g}}M_+\wedge i_{\ast} K (n))^G$, is a weak equivalence. Since $\lambda_{G,M}\wedge K(n)=\aaa_{G,M}^G$ by proposition \ref{alphalambda}, this proves part (b). Finally, as pointed out in the remark, part (c) follows from the fact that $K(n)$-orientability is the same as the ordinary orientability for $p>2$ \cite{Rudyak}.
\end{proof}

\section{Duality from the viewpoint of stacks}\label{sec:stack}

As mentioned in the introduction, the constructions of the spectrum maps (\ref{swd}) and (\ref{bgd}) underlying the Spanier-Whitehead duality for manifolds and $K(n)$-duality for classifying spaces of finite groups have common ingredients. Both maps are special cases of $\lambda_{G,M}$ in theorem \ref{mainquotient}. In this section, we will study $\lambda_{G,M}$ from the point of view of stacks. We will show that it descends to a map of stack andba can be interpreted as the Spanier-Whitehead duality map for differentiable stacks. To do this we need a version of Pontryagin-Thom map for stacks developed in \cite{Ebert}. We will first review some basic facts about stacks. A more detailed introduction on this subject can be found in \cite{H},\cite{NoohiTopSt}. For the theory of homotopy type of topological stacks, see \cite{Ebert} and \cite{Noohi}.

\subsection{Some basic facts of stacks}
Let $\Diff$ be the category of smooth manifolds and smooth maps and $\textbf{Gpd}$ be the category of groupoids and natural transformations.
Roughly speaking, a stack is a contravariant pseudo-functor $\mathfrak{X}:\Diff\to\textbf{Gpd}$ which satisfies sheaf-like properties so that we can glue compatible objects and morphisms. One can also define morphisms between stacks and 2-morphisms between morphisms. They form the 2-category of stacks. Indeed, all the 2-morphisms are invertible. Two stacks $\mfX,\mfY$ are equivalent if there exist morphisms $u:\mfX\to\mfY$, $v:\mfY\to\mfX$ such that $v\circ u$ is 2-isomorphic to $1_{\mfX}$ and $u\circ v$ is 2-isomorphic to $1_{\mfY}$. For simplicity. we will denote 2-commutative diagrams and 2-pullbacks by commutative diagrams and pullbacks respectively. 

By Yoneda embedding, every smooth manifold $M$ defines a stack. More generally, every Lie groupoid $\mathbbm{X}= [\mathbbm{X}_1 \rightrightarrows
\mathbbm{X}_0]$ gives rise to a differentiable stack $[ \mathbbm{X}_0 /\mathbbm{X}_1]$, where $[ \mathbbm{X}_0 /\mathbbm{X}_1](N)$ is the groupoid of principal $\mathbbm{X}$-bundles over $N$. If $\mathbbm{X}$ is the action groupoid $[G \times M\rightrightarrows M]$ associated to a Lie group action on a smooth manifold, a $\mathbbm{X}$-bundle over $N$ consists of a principal $G$-bundle $P\to N$ and a $G$-map $f:P\to M$. The differentiable stack arising from this Lie groupoid $[G \times M\rightrightarrows M]$ is also denoted by $[M/G]$. 
%For $M = pt$ is a point, $[pt / G] (N)$ is simply the groupoid of principal $G$-bundles over $N$ and bundle maps between them. $[pt/G]$ is known as the classifying stack of principal $G$-bundles and $pt\to[pt/G]$ is the associated universal principal $G$-bundle

The smooth structure of Lie groupoids allows one to define tangent stacks for differentiable stacks. For the case $\mfX\simeq[M/G]$, where the $G$-action is almost free, the tangent stack  $T\mfX\to\mfX$ can be described in terms of a $G$-vector bundle over $M$. At each point $x\in M$, the derivative of the map $G\to M$ given by $g\mapsto gx$ defines a linear map $\mathfrak{g}\to T_xM$. This linear map is injective since the $G$-action is almost free. By varying $x\in M$, the images of these linear maps form a $G$-subbundle of $TM$. By abuse of notation, we denote this subbundle by $\mathfrak{g}$. The tangent stack $T\mfX\to \mfX$ is the $G$-orbit of the quotient bundle $TM/\mathfrak{g}\to M$. %It is a well-defined notion of quotient stacks in the sense that if $\mfX$ is also equivalent to $[N/H]$, where $H$ is a compact Lie group and $N$ is a smooth $H$-manifold with finite stabilizers, then there is a canonical equivalence $[(TM/\mathfrak{g})/G]\simeq [(TN/\mathfrak{h})/H]$ over $\mfX$. 

To discuss the homotopy type of differentiable stacks, we have to consider stacks defined over the site $\Top$, the category of compactly generated spaces and continuous maps. Analogous to differentiable stacks, topological stacks are stacks over $\Top$ arising from topological groupoids. A differentiable stack can be regarded as a topological stack by neglecting the smooth structure of the associated Lie groupoid. The homotopy type of a topological stack $\mfX$ is described by a morphism $\eta_{\mfX}:\Ho(\mfX)\to\mfX$, where $\Ho(\mfX)$ is a topological space and $\eta_{\mfX}$ is a universal weak equivalence in the sense that if $Y\to\mfX$ is a morphism from a space $Y$, the pullback $\Ho(\mfX)\times_{\mfX}Y\to Y$ is a weak equivalence of spaces. One can define $\Ho$ and $\zeta_{\mfX}$ in a functorial way \cite{Ebert},\cite{Noohi}. $\Ho ([\mathbbm{X}_0 / \mathbbm{X}_1])$ is given by $B \mathbbm{X}$, the classifying space of $\mathbbm{X}$. In particular, if $\mfX=[M/G]$, $\Ho(\mfX)=M_{hG}=EG\times_GM$ is the homotopy orbit of $M$ and $\zeta_{\mfX}$ is defined by the principal $G$-bundle $EG\times M\to EG\times_G M$ and the $G$-map $\pi_2:EG\times M\to M$. An equivalence of stacks $[ \mathbbm{X}_0 / \mathbbm{X}_1]\simeq[ \mathbbm{Y}_0 / \mathbbm{Y}_1]$ would induce a canonical weak equivalence of their classifying spaces $B\mathbbm{X}\simeq B \mathbbm{Y}$. We will describe this canonical weak equivalence for the cases of quotient stacks in section \ref{sec:pfgmhn}.

\subsection{Proof of proposition \ref{prop:gmhn}}\label{sec:pfgmhn}

We first describe the canonical weak equivalences (\ref{weho}) and (\ref{wethom}) in proposition \ref{prop:gmhn}. Let $G$,$H$ be compact Lie groups. Suppose $M$ is a smooth closed $G$-manifold and $N$ is a smooth closed $H$-manifold such that their associated quotient stacks are equivalent. Let $\mfX\simeq[M/G]\simeq[N/H]$. The product of the atlases $M\to\mfX$ and $N\to\mfX$ is a principal $(G\times H)$-bundle $M\times N \to \mfX\times \mfX$. The pullback bundle under the diagonal map $\Delta_{\mfX}:\mfX\to\mfX^2$ defines another atlas $P\to\mfX$, where $P$ is a closed $(G\times H)$-manifold with $\mfX\simeq[P/(G\times H)]$. $P$ also fits into the pullback diagram
\begin{equation}\label{pullbackmpn}
\begin{aligned}
\xymatrix{
P\ar^{\pi_N}[r]\ar^{\pi_M}[d]& N\ar[d]\\
M\ar[r]&\mfX
}
\end{aligned}
\end{equation}
where the horizontal maps are principal $G$-bundles and the vertical maps are principal $H$-bundles. $\pi_M$ and $\pi_N$ is equivariant with respect to the projection map $G\times H\to G$ and $G\times H\to H$ respectively. The maps between the corresponding action groupoids satisfy the following commutative diagram
\begin{equation}\label{diagonalmpn}
\xymatrixrowsep{0.45cm}
\xymatrixcolsep{0.4cm}
\begin{aligned}
\xymatrix{
G'\times H'\times P_2\ar[rr]^-{\pi_{H'}\times\pi_N}\ar[rd]^(0.7){\pi_{G'}\times\pi_M}\ar[ddd]^(0.6){\Delta_{G\times H, P}} & & H'\times N_2\ar[ddd]^(0.6){\Delta_{H,N}}|(.38)\hole \ar[rd] &\\
  & G' \times M_2 \ar[rr]\ar[ddd]^(0.4){\Delta_{G,M}} & & \mfX\ar[ddd]^(0.4){\Delta_{\mfX}}\\\\
P_1\times P_2\ar^(0.4){\pi_N^2}[rr]|(.54)\hole\ar^(0.6){\pi_M^2}[rd] & & N_1\times N_2\ar[rd] &\\
  & M_1\times M_2 \ar[rr] & & \hspace{1mm}\mfX^2,}
\end{aligned}
\end{equation}
where each of its faces is a pullback square. The vertical maps $\Delta_{G,M}$, $\Delta_{H,N}$ and $\Delta_{G\times H,N}$ are pullback of the diagonal map $\mfX\to\mfX^2$ along the atlases of $M^2,N^2$ and $P^2$ of $\mfX^2$. The maps between these action groupoids induce weak equivalences
\begin{equation}\label{3ho}
M_{hG}\xleftarrow{\simeq}P_{h(G\times H)}\xrightarrow{\simeq}N_{hH} 
\end{equation}
between the associated classifying spaces. Note that the maps $G'\times M\to\mfX$ and $H'\times N_2\to\mfX$ of the top square of diagram (\ref{diagonalmpn}) define two atlases of $\mfX$, and their pullback $G'\times H'\times P$ is another one. There are equivalences of stacks 
\[[(G'\times M_2)/G^2]\xleftarrow{\simeq}[(G'\times H'\times P_2)/(G\times H)^2]\xrightarrow{\simeq}[(H'\times N_2)/H^2] \]
which induce weak equivalences 
\begin{equation}\label{3hoho}
(G'\times M_2)_{hG^2}\xleftarrow{\simeq} (G'\times H'\times P_2)_{h(G\times H)^2} \xrightarrow{\simeq} (H'\times N_2)_{hH^2}
\end{equation}
of homotopy orbits.

If furthermore, the $G$-action on $M$ (and hence the $H$-action on $N$) is almost free, then $\mfX$ is a Deligne-Mumford stack. The pullback of $T\mfX\to\mfX$ along $EG\times_GM\to\mfX$ is the vector bundle $EG\times_G(TM/\mathfrak{g})\to EG\times_GM$. Similarly, we obtain vector bundles $EH\times_H(TN/\mathfrak{h})\to EH\times_HN$ and $(EG\times EH)\times_{(G\times H)}(TP/\mathfrak{g}+\mathfrak{h})\to (EG\times EH)\times_{(G\times H)}P$. Moreover, the last one is the pullback of the other two along the maps in (\ref{3ho}). These induce weak equivalences
\begin{equation}\label{3thom}
(M_{hG})^{\mathfrak{g}-TM}\xleftarrow{\simeq}(P_{h(G\times H)})^{\mathfrak{g}+\mathfrak{h}-TP}\xrightarrow{\simeq}(N_H)^{\mathfrak{h}-TN}
\end{equation}
between the corresponding Thom spectra.

The maps (\ref{3ho}) and (\ref{3thom}) are the canonical weak equivalence (\ref{weho}) and (\ref{wethom}) described in the introduction. We now prove that they commute with $\lambda_{G,M}$ and $\lambda_{H,N}$ as in proposition \ref{prop:gmhn}. 

\begin{proof}[Proof of proposition \ref{prop:gmhn}]
Let $M,N,P$ be atlases of $\mfX$ as in diagram (\ref{pullbackmpn}). To prove proposition \ref{prop:gmhn}, we will show that the diagram
\begin{equation}\label{lambdampn}
\begin{aligned}
\xymatrix{
\Sigma^{\infty}{M_{hG}}_+\ar^{\lambda_{G,M}}[rr] && F((M_{hG})^{\mathfrak{g}-TM},\SSS)\ar^{\simeq}[d]\\
\Sigma^{\infty}{P_{h(G\times H)}}_+ \ar^-{\lambda_{G\times H,P}}[rr] \ar_{\simeq}[u]\ar^{\simeq}[d]&& F((P_{h(G\times H)})^{\mathfrak{g}+\mathfrak{h}-TP},\SSS)\\
\Sigma^{\infty}{N_{hH}}_+\ar^{\lambda_{H,N}}[rr]&& F((N_{hH})^{\mathfrak{h}-TN},\SSS).\ar_{\simeq}[u]\\
}
\end{aligned}
\end{equation}
commutes. Here the vertical maps are induced by (\ref{3ho}) and (\ref{3thom}). By the definition of $\lambda_{G,M}$, the commutativity of (\ref{lambdampn}) is an easy consequence of that of the following diagram 
\begin{equation}\label{PTmpn}
\begin{aligned}
\xymatrix{
\Sigma^{\infty}{M_{hG}}_+\wedge(M_{hG})^{-TM+\mathfrak{g}}\ar^-{\beta_{G,M}/G^2}[r] & \Sigma^{\infty}{(G'\times M_2)_{hG^2}}_+\\
\Sigma^{\infty}{P_{h(G\times H)}}_+\wedge(P_{h(G\times H)})^{-TP+\mathfrak{g}+\mathfrak{h}}\ar^-{\stackrel{\beta_{G\times H,P}/(G\times H)^2}{\stackrel{}{}}}[r] \ar_{\simeq}[u]\ar^{\simeq}[d]& \Sigma^{\infty}{(G'\times H'\times P)_{h(G\times H)^2}}_+\ar_{\simeq}[u]\ar^{\simeq}[d]\\
\Sigma^{\infty}{N_{hH}}_+\wedge(N_{hH})^{-TN+\mathfrak{h}}\ar^-{\beta_{H,N}/H^2}[r]& \Sigma^{\infty}{(H'\times N_2)_{hH^2}}_+\\
}
\end{aligned}
\end{equation}
where the left vertical maps are induced by (\ref{3ho}) and (\ref{3thom}) and the right vertical maps are induced by (\ref{3hoho}). The proof of the commutativity of (\ref{PTmpn}) is similar to that of proposition \ref{MackeyGHG}. We relate $\beta_{G\times H,P}$ with each of $\beta_{G,M}$ and $\beta_{H,N}$ by picking different data in the construction of $\beta_{G\times H,P}$. 

Let $\overline{U},U',U_0$ be $(G\times H)^2$-universes, $\overline{i},i',j'$ be inclusions of universes and $\iota_G:G'\to W_G$ be a $G^2$-embedding as in the proof of proposition \ref{MackeyGHG}. Pick a $G$-vector bundle $Q\to M$ such that $Q\oplus TM$ is isomorphic to the trivial bundle $M\times V$ for some $G$-representation $V$. Then $\pi_M^{\ast}Q\oplus TP\cong \pi_M^{\ast}Q\oplus \pi_M^{\ast}TM\oplus \mathfrak{h}\cong P\times (V\oplus\mathfrak{h})$. Let $$\iota_1:G'\times H'\times P_2\to P_1\times (\pi_M^{\ast}Q)_2 \times W_G \times \mathfrak{g}_2,$$ 
$$\iota_{G\times H,P}:G'\times H'\times P_2\to P_1\times (\pi_M^{\ast}Q)_2 \times W_G \times \mathfrak{g}_2\times \mathfrak{h}_2$$ be given by $\iota_1(g',h',x)=((g',h')x,x,\iota_G(g'),0)$ and $\iota_{G\times H,P}(g',h',x)=(\iota_1(g',h',x),0)$. Note that $W_G,V_2\subset U'$. Applying Pontryagin-Thom construction to $\iota_1$ and $\iota_{G\times H,P}$ followed by desuspension $\Sigma^{-W_G-V_2}$ and $\Sigma^{-W_G-V_2-\mathfrak{h}_2}$ respectively yields morphisms 
$$\Delta_1^!:i'_{\ast}\Sigma^{\infty}{P_1}_+ \wedge (P^{-TP+\mathfrak{g}+\mathfrak{h}})_2 \to i'_{\ast}\Sigma^{\infty}(G'\times H'\times P_2)_+$$ 
in $(G\times H)^2\mathcal{S}U'$ and
$$\Delta_{G\times H,P}^!:\overline{i}_{\ast}\Sigma^{\infty} {P_1}_+\wedge (P^{-TP+\mathfrak{g}+\mathfrak{h}})_2 \to \overline{i}_{\ast}\Sigma^{\infty} (G'\times H'\times P_2)_+$$ 
in $(G\times H)^2\mathcal{S}\overline{U}$. Clearly $j'_{\ast}\Delta^!_1=\Delta_{G\times H, P}^!$. 
By theorem \ref{isochangeofu} and the $(G\times H)^2$-freeness of $EG^2\times P_1\times P_2$, there exist unique homotopy classes of morphisms
$$\beta_1:\Sigma^{\infty}(EG_1\times P_1)_+\wedge (EG_2\times P_2)^{-TP_2+\mathfrak{g}_2+\mathfrak{h}_2}\to \Sigma^{\infty}(EG^2\times G'\times H' \times P_2)_+$$ 
\begin{equation*}
\begin{split}
\beta_{G\times H,P}:&\Sigma^{\infty}(EG_1\times EH_1\times P_1)_+\wedge (EG_2\times EH_2\times P_2)^{-TP_2+\mathfrak{g}_2+\mathfrak{h}_2}\\
&\to \Sigma^{\infty}(EG^2\times EH^2\times G' \times H' \times G_2)_+
\end{split}
\end{equation*}
in $(G\times H)^2\mathcal{S}U_0$ such that $i'_{\ast}\beta_1=EG^2_+\wedge \Delta^!_1$ and $\overline{i}_{\ast}\beta_{G\times H,P}=EG_+^2\wedge EH_+^2\wedge\Delta_{G\times H,P}^!$. This implies 
$$\overline{i}_{\ast}\beta_{G\times H,P}= EG^2_+\wedge EH_+^2\wedge j'_{\ast}\Delta^!_1 =EH_+^2\wedge j'_{\ast}i'_{\ast}\beta_1=\overline{i}_{\ast} (EH^2_+\wedge \beta_1)$$ and so $\beta_{G\times H,P} =EH^2_+\wedge \beta_1$.

On the other hand, under the identifications
$$(G'\times H'\times P_2)/(1\times H)^2 \cong G'\times M_2 \text{ and } P/H\cong M$$ 
induced by $\pi_{G'}\times \pi_M$ and $\pi_M$, $\iota_1/(1\times H)^2$ is a $G^2$-embedding $G'\times M_2\to M_1\times Q_2\times W_G \times \mathfrak{g}_2$ over $\Delta_{G,M}$. It follows that $\Delta^!_1/(1\times H)^2=\Delta_{G,M}^!$ and $\beta_1/(1\times H)^2=\beta_{G,M}$. 
%$$P^{-TP+\mathfrak{g}+\mathfrak{h}}/H\cong M^{-\pi_M^{\ast}TM+\mathfrak{g}}$$

Therefore, $\beta_{G\times H,P}/(G\times H)^2=(EH^2_+\wedge \beta_1)/(G\times H)^2$ and $\beta_{G,M}/G^2=\beta_1/(G\times H)^2$. This proves the commutativity of the upper half of the diagram in the proposition. The commutativity of the bottom half of the diagram is proved similarly.
\end{proof} 

\subsection{Stack version of Spanier-Whitehead type construction}

As mentioned in the introduction, proposition \ref{prop:gmhn} allows us to interpret $\lambda_{G,M}$ in (\ref{egmge}) as a map $\lambda_{\mfX}$ of the stack $\mfX=[M/G]$. Indeed, it is also possible to describe the construction of $\lambda_{G,M}$ in terms of $\mfX$. We will look at the construction of $\lambda_{G,M}$ from the point of view of stacks and compare it with that of the Spanier-Whitehead duality for manifolds. 

Recall that the Spanier-Whitehead duality for manifolds arises from the Pontryagin-Thom maps for diagonal maps of manifolds. In \cite{Ebert} Ebert and Giansiracusa generalized the construction of Pontryagin-Thom maps to a certain class of maps between differentiable stacks. For such a map $f:\mfX\to\mfY$ with stable normal bundle $v(f)$, their construction yields a map 
\begin{equation}\label{ebertPT}
\Ho(\mfY)\to \Omega^{\infty}\Ho(\mfX)^{v(f)}.
\end{equation}
We will apply Pontryagin-Thom construction to the diagonal map $\Delta_{\mfX}:\mfX\to\mfX^2$ of a differentiable Deligne-Mumford stack $\mfX\simeq[M/G]$. In this case, the normal bundle $v(\Delta_{\mfX})$ over $\mfX$ is isomorphic to $T\mfX \cong \Delta_{\mfX}^{\ast}(\mfX\times T\mfX)$, the pullback of the vector bundle $\mfX\times T\mfX\to \mfX^2$ along $\Delta_{\mfX}$. By a slight variation of the construction of (\ref{ebertPT}), we define a map
\begin{equation}\label{PTstackdiagonal}
\Sigma^{\infty}\Ho(\mfX)_+\wedge \Ho(\mfX)^{-T\mfX} \to \Sigma^{\infty}\Ho(\mfX)_+
\end{equation}
of spectra. This map is analogous to (\ref{Mdiagpt}) in the construction of Spanier-Whitehead duality for manifolds.

Let $\mfX$ be a differentiable Deligne-Mumford stack. Suppose $\mfX\simeq [M/G]$ for some smooth closed manifold $M$ with an almost free action by a compact Lie group $G$. Recall that for a $G$-space $X$ and $i=1,2$, we denote by $X_i$ the $G^2$-space $\pi_i^{\ast}X$, where $\pi_i:G^2\to G$ is the projection to the $i$-th factor. Also, the $G^2$-manifold $G'$ is a diffeomorphic copy of $G$ with $G^2$-action $(g_1,g_2)g'=g_1g'g_2^{-1}$. We have atlases $M\to\mfX$ and $M_1\times M_2\to\mfX^2\simeq[(M_1\times M_2)/G^2]$. The pullback of the atlas $M_1\times M_2\to \mfX^2$ along the diagonal map $\Delta_{\mfX}:\mfX\to\mfX^2$ is another atlas $G' \times M_2\to \mfX$ and $\mfX\simeq[(G'\times M_2)/G^2]$. There is a pullback diagram
\[
\xymatrix{
G'\times M \ar[d]\ar^-{\Delta_{G,M}}[r] &M_1\times M_2\ar[d]\\
\mfX\ar^{\Delta_{\mfX}}[r]&\hspace{2mm}\mfX^2.
}
\]
The diagonal map $\Delta_{\mfX}$ can be represented by the $G^2$-orbit of the map $\Delta_{G,M}:G'\times M_2\to M_1 \times M_2$. Unlike the case of manifolds, this diagonal map $\Delta_{\mfX}$ is not an embedding in general. Nevertheless, since $M$ and $G'$ can be embedded into a $G$-representation and $G^2$-representation respectively, $\mfX$ can be embedded into a vector bundle of $\mfX^2$ over $\Delta_{\mfX}$ as we will describe below. 

Suppose $\iota_{G,M}:G'\times M_2\to M_1\times Q_2 \times W \times \mathfrak{g}_2$ is the $G^2$-embedding (\ref{iota}). It has normal bundle $(G'\times M_2)\times V_2\times W$. By tubular neighborhood theorem, there exists a $G^2$-map $$\eta_{G,M}:(G'\times M_2)\times V_2\times W\to M_1\times Q_2 \times W \times \mathfrak{g}_2$$ which embeds the normal bundle of $\iota_{G,M}$ as an open subset of $M_1\times Q_2 \times W \times \mathfrak{g}_2$ with zero section being $\iota_{G,M}$. Let $\mathfrak{Q}\to \mfX,\mathfrak{V}\to \mfX,\mathfrak{g}\to \mfX$ be the $G$-orbit of $Q\to M,M\times V\to M,M\times \mathfrak{g}\to M$ respectively and $\mathfrak{W}\to \mfX^2$ be the $G^2$-orbit of $W\to M_1\times M_2$. On passage to the $G^2$-orbit of $\iota_{G,M}$ we obtain an embedding $\iota_{\mfX}:\mfX\to\mathfrak{Q}_2\oplus \mathfrak{W}\oplus\mathfrak{g}_2$ over $\Delta_{\mfX}$ with normal bundle $\Delta_{\mfX}^{\ast}(\mathfrak{V}_2\oplus\mathfrak{W})\to\mfX$. The $G^2$-orbit of $\eta_{G,M}$ gives a tubular neighborhood $\eta_{\mfX}:\Delta_{\mfX}^{\ast}(\mathfrak{V}_2\oplus\mathfrak{W})\to \mathfrak{Q}_2\oplus \mathfrak{W}\oplus\mathfrak{g}_2$ of $\iota_{\mfX}$. The $G^2$-equivariant maps $\Delta_{G,M},\iota_{G,M}$ and $\eta_{G,M}$ descend to morphisms of stacks which satisfy the following commutative diagram 
\begin{equation}\label{diagramPTmfx}
\begin{aligned}
\xymatrix{
\Delta_{\mfX}^{\ast}(\mathfrak{V}_2\oplus\mathfrak{W})\ar@{^{(}->}[rr]^-{\eta_{\mfX}}&&\mathfrak{Q}_2\oplus \mathfrak{W}\oplus\mathfrak{g}_2\ar[d]\\
\mathfrak{X}\ar[rr]^{\Delta_{\mfX}}\ar[rru]^{\iota_{\mfX}}\ar^s[u] && \mathfrak{X}^2.
}
\end{aligned}
\end{equation}

The homotopy type of $\mfX^2$ is represented by an universal weak equivalence $\zeta_{\mfX^2}:\Ho(\mfX^2)\to\mfX^2$. The pullback of $\zeta_{\mfX^2}$ along $\Delta_{\mfX}$ is another universal weak equivalence and can be regarded as the homotopy type $\zeta_{\mfX}:\Ho(\mfX)\to\mfX$ of $\mfX$. By taking pullback along $\zeta_{\mfX^2}$, we obtain a commutative diagram
\begin{equation*}%\label{diagramPTho}
\begin{aligned}
\xymatrix{
\zeta_{\mfX}^{\ast}\Delta_{\mfX}^{\ast}(\mathfrak{V}_2\oplus\mathfrak{W})\ar@{^{(}->}[rr]^-{\Ho(\eta_{\mfX})}&&\zeta_{\mfX}^{\ast}(\mathfrak{Q}_2\oplus \mathfrak{W}\oplus\mathfrak{g}_2)\ar[d]\\
\Ho(\mathfrak{X})\ar[rr]^{\Ho(\Delta_{\mfX})}\ar[rru]^{\Ho(\iota_{\mfX})}\ar[u] && \Ho(\mathfrak{X}^2)
}
\end{aligned}
\end{equation*}
of spaces from (\ref{diagramPTmfx}). The open embedding $\Ho(\eta_{\mfX})$ defines a based map
$$
\Ho(\mfX^2)^{\mathfrak{Q}_2\oplus \mathfrak{W}\oplus\mathfrak{g}_2}\to \Ho(\mfX)^{\Delta_{\mfX}^{\ast}(\mathfrak{V}_2\oplus\mathfrak{W})}
$$
between the two Thom spaces. Note that $T\mfX\oplus \mathfrak{g} \oplus \mathfrak{Q}=\mfV$. Hence, a twisted version of desuspension by the vector bundle $\mfV_2\oplus\mathfrak{W}\to\mfX^2$ yields (\ref{PTstackdiagonal}).
The map $\beta_{G,M}/G^2$ in (\ref{betaovergg}) is a homotopy representative of (\ref{PTstackdiagonal}) arising from $\mfX\simeq[M/G]$. Diagram (\ref{PTmpn}) shows the explicit canonical weak equivalence between $\beta_{G,M}/G^2$ and another representative $\beta_{H,N}/H^2$ of (\ref{PTstackdiagonal}) arising from $\mfX\simeq[N/H]$. 

By further composing (\ref{PTstackdiagonal}) with the collapse map $\Sigma^{\infty}\Ho(\mfX)_+ \to \SSS$ and taking adjoint, we obtain 
\begin{equation*}
\lambda_{\mfX}:\Sigma^{\infty}\Ho(\mfX)_+\to F(\Ho(\mfX)^{-T\mfX}, \SSS),
\end{equation*}
which is the map (\ref{stackswd}) in corollary \ref{mainstack}. 
For the case where $\mfX$ is a manifold, the construction of $\lambda_{\mfX}$ described above reduces to that of the classical Spanier-Whitehead duality.

%The total spaces of the vector bundle $\zeta_{\mfX}^{\ast}\Delta_{\mfX}^{\ast}(\mathfrak{V}\oplus\mathfrak{W})\to \Ho(\mfX)$ and $\zeta_{\mfX^2}^{\ast}(\mathfrak{Q}_2\oplus \mathfrak{W}\oplus\mathfrak{g}_2)\to \Ho(\mfX^2)$ is $EG^2\times_{G^2}(G'\times M_2\times V_2 \times W)$ and $EG^2\times_{G^2}(M_1\times Q_2\times V_2 \times W \times \mathfrak{g}_2)$ respectively. Also, $\Ho(\Delta_{\mfX})=EG^2\times_{G^2}\Delta_{G,M}$ and $\Ho(\iota_{\mfX})=EG^2\times_{G^2}\iota_{G,M}$.

%We would also like to explain the relationship between the horizontal maps in each of the diagrams (\ref{3coverings}) and (\ref{3PT}). Recall that $G$ is a compact Lie group and $H$ is a finite subgroup of $G$. Let $\mfX$ be the classifying stack $[pt/H]$. The universal principal $H$-bundle $pt\to\mfX$ is an atlas of $\mfX$. The principal $H$-bundle $G\to G/H$ defines another atlas $G/H \to \mfX$, whose associated Lie groupoid $[(G/H)\times_{\mfX}(G/H)=G\times (G/H) \rightrightarrows G/H]$ has structure maps given by projection and $G$-action on $G/H$. The pullback of these two atlases is another atlas $pt\times_{\mfX}G/H=G\to \mfX$.% which corresponds to the trivial principal $H$-bundle over $G$. The associated Lie groupoid for this atlas $G\to\mfX$ is the action groupoid of the $(G\times H)$-action $(g,h)x=gxh^{-1}$ on $G$. Hence, $\mfX\simeq[pt/H]\simeq[(G/H)/G]\simeq[G/(G\times H)]$. 

\section{Some examples on finite groups}

In this section we look at the $K(n)$-duality of stacks defined by finite groups. In section \ref{sec:pftransversecap} we study the relation between the $K(n)$-duality of a finite group with that of its subgroups and prove theorem \ref{thm:transversecap}. Some calculations of the intersection product and $K(n)$-fundamental class of cyclic $p$-groups will be given in section \ref{sec:cyc}. 

\subsection{Proof of theorem \ref{thm:transversecap}}\label{sec:pftransversecap}
First let us fix our notations. Let $i:H\to G$ be an inclusion of finite groups. The same symbol $i$ is also used to denote the induced morphisms $[pt/H]\to[pt/G]$ of stacks, $BH\to BG$ of homotopy types and $\Sigma^{\infty}BH_+\to \Sigma^{\infty}BG_+$ of spectra. The transfer map $\Sigma^{\infty}BG_+\to\Sigma^{\infty}BH_+$ and its induced map $K(n)_{\ast}(BG)\to K(n)_{\ast}(BH)$ will both be denoted by $i^!$. We write $\lambda_H:\Sigma^{\infty} BH_+ \longrightarrow F (\Sigma^{\infty}BH_+, \mathbb{S})$ for the map (\ref{stackswd}) in theorem \ref{mainstack} for the case $\mfX=[pt/H]$. Recall from definition \ref{fc} that the $K(n)$-fundamental class of $[pt/H]$ is defined to be the pre-image of 1 under $(\lambda_H)_{\ast}:K (n)_{\ast} (BH) \cong K (n)^{- \ast}(BH)$. Let $[BH]$ denote both the $K(n)$-fundamental class of $BH$ in $K(n)_0(BH)$ and its image in $K(n)_0(BG)$ under $i_{\ast}$. 

Next, we look at intersection of subgroups. Recall from definition \ref{def:transv} in the introduction that two subgroups $H$ and $K$ of a finite group $G$ are said to intersect transversely if $HK:=\{hk|h\in H,k\in K\}=G$. 
\begin{lemma}\label{l:transverseequivalent}
Let $G$ be a finite group and $H,K$ be subgroups of $G$. The following are equivalent:
\begin{enumerate}[(i)]
	\item $H,K$ intersect transversely; 
	\item $|H||K|/|H\cap K|=|G|$;
	\item The quotient map $\pi:G/(H\cap K)\to G/H \times G/K$ is an isomorphism of left $G$-sets.
\end{enumerate}
\end{lemma}
\begin{proof}
We will show each of (i) and (ii) is equivalent to (ii) by counting argument. To show (i)$\Leftrightarrow$(iii), let $H\times K$ acts on $HK$ by $(h,k)x=hxk^{-1}$. It is clear that it is a transitive action with the stabilizer of $e\in HK$ equals to $\{(a,a^{-1})|a\in H\cap K\}$, which has cardinality $|H\cap K|$. Hence, $|H||K|/|H\cap K|=|HK|$. As a result, $HK=G$ if and only if $|H||K|/|H\cap K|=|G|$.

For (ii)$\Leftrightarrow$(iii), it is clear that $\pi$ is injective. Hence, $\pi$ is a bijection, and hence an isomorphism of $G$-sets if and only if $|G/(H\cap K)|=|G/H \times G/K|$. The last condition is equivalent to (ii).
\end{proof}

An important tool for us is the Mackey property. Let $\pi:P\to Y$ be a finite covering. Consider $f:X\to Y$ and the pullback diagram
\[\xymatrix{
f^{\ast}P\ar[r]^g\ar[d]^{f^{\ast}\pi}& P\ar[d]^{\pi}\\
X\ar[r]^f& Y}\]
The Mackey property states that the diagram
\[\xymatrix{
\Sigma^{\infty}(f^{\ast}P)_+\ar[rr]^{\Sigma^{\infty}g_+}&& \Sigma^{\infty}P_+\\
\Sigma^{\infty}X_+\ar[rr]^{\Sigma^{\infty}f_+}\ar[u]^{(f^{\ast}\pi)^!}&& \Sigma^{\infty}Y_+\ar[u]^{(\pi)^!}}\]
commutes.

We will mainly apply the Mackey property to covering maps arising from inclusions of subgroups of finite groups. 

\begin{proposition}\label{prop:transversepbmackey}
Suppose $H$ and $K$ are transverse subgroups of a finite group $G$. Then there is a pullback diagram
\[\xymatrix{
[pt/(H\cap K)]\ar[r]^-p\ar[d]^q& [pt/K]\ar[d]^j\\
[pt/H]\ar[r]^i& [pt/G]}\]
with all the maps in it are induced by inclusions of groups. There is also a commutative diagram of spectra
\[
\begin{aligned}
\xymatrix{
\Sigma^{\infty}B(H\cap K)_+\ar[r]^-p& \Sigma^{\infty}BK_+\\
\Sigma^{\infty}BH_+\ar[r]^i\ar[u]^{q^!}& \Sigma^{\infty}BG_+\ar[u]^{j^!}}
\end{aligned}
\]
\end{proposition}

\begin{proof}
The map $i:[pt/H]\to[pt/G]$ and $j:[pt/K]\to[pt/G]$ is the $G$-orbit of $G/H\to pt$ and $G/K\to pt$ respectively. Hence, we have $[pt/H]\times_{[pt/G]}[pt/K]\simeq [(G/H \times G/K)/G]$. By lemma \ref{l:transverseequivalent}, it is equivalent to $[(G/(H\cap K))/G]\simeq [pt/(H\cap K)]$. It is easy to see $p,q$ are induced by inclusions. The commutativity of the second diagram follows from the first diagram and Mackey property.
\end{proof}

The Mackey property can be used to show the relation between $\lambda_G$ and $\lambda_H$ for $H\subset G$.

\begin{proposition}\label{prop:iidualcomm}
Suppose $H$ is a subgroup of a finite group $G$ and $i:H\to G$ is the inclusion. Then there is a commutative diagram
\[
\xymatrix{
K(n)_{\ast}(BH)\ar[r]^{(\lambda_H)_{\ast}}&K(n)^{-\ast}(BH)\\
K(n)_{\ast}(BG)\ar[u]^{i^!}\ar[r]^{(\lambda_G)_{\ast}}&K(n)^{-\ast}(BG)\ar[u]^{i^*}
}
\]
In particular, $i^!([BG])=[BH]$.
\end{proposition}
\begin{proof}
Consider the following commutative diagram of group monomorphisms and its induced diagram on classifying spaces
\[
\xymatrix{
H\ar[r]^i\ar[d]^{(i,1)}&G\ar[d]^{\Delta_G}\\
G\times H\ar[r]^{1\times i}&G\times G
}
\hspace{1cm}
\xymatrix{
BH\ar[r]^i\ar[d]^{(i,1)}&BG\ar[d]^{\Delta_G}\\
BG\times BH\ar[r]^{1\times i}&BG\times BG
}
\]
Since $G\times H, \Delta_G(G)$ are transverse subgroups of $G\times G$, proposition \ref{prop:transversepbmackey} implies that the second diagram is a pullback diagram. Also, the left vertical map of the first diagram is equal to the composite $H\xrightarrow{\Delta_H}H\times H\xrightarrow{i\times 1}G\times H$ of monomorphisms. Hence, $(i,1)^!=\Delta_H^!(i\times 1)^!=\Delta_H^!(i^!\times 1)$. By the Mackey property, there is a commutative diagram
\begin{equation}\label{mackeyGH}
\begin{aligned}
\xymatrix{
\Sigma^{\infty}BH_+\ar[r]^i&\Sigma^{\infty}BG_+\\
\Sigma^{\infty}BH_+\wedge \Sigma^{\infty}BH_+\ar[u]^{\Delta_H^!}&\\
\Sigma^{\infty}BG_+\wedge \Sigma^{\infty}BH_+\ar[u]^{i^!\times 1}\ar[r]^{1\times i}&\Sigma^{\infty}BG_+\wedge \Sigma^{\infty}BG_+\ar[uu]^{\Delta_G^!}
}
\end{aligned}
\end{equation}

Return to the diagram in the proposition and consider the two composites in it. The map $(\lambda_H)_{\ast}i^!$ is obtained from the composite
\[\Sigma^{\infty}BG_+\wedge \Sigma^{\infty}BH_+\xrightarrow{\Delta_H^!(i^!\times 1)}\Sigma^{\infty}BH_+\xrightarrow{\epsilon_H}\SSS\]
by taking adjoint and $K(n)$-homology.
Similarly , the other map $i^{\ast}(\lambda_G)_{\ast}$ is obtained from the composite
\[\Sigma^{\infty}BG_+\wedge \Sigma^{\infty}BH_+\xrightarrow{\Delta_G^!(1\times i)}\Sigma^{\infty}BG_+\xrightarrow{\epsilon_G}\SSS\]
by the same procedure. By the commutative diagram (\ref{mackeyGH}) and the fact that $\epsilon_G i=\epsilon_H$, the two maps $\Sigma^{\infty}BG_+\wedge \Sigma^{\infty}BH_+\to\SSS$ above are equal, so are $(\lambda_H)_{\ast}i^!$ and $i^{\ast}(\lambda_G)_{\ast}$.

The last part of the proposition follows easily from the commutativity of the diagram in the proposition:
$$i^!([BG])=(\lambda_H)_{\ast}^{-1}i^{\ast}(\lambda_G)_{\ast}([BG])=(\lambda_H)_{\ast}^{-1}i^{\ast}(1)=(\lambda_H)_{\ast}^{-1}(1)=[BH].$$
\end{proof}

Let $\alpha\in K(n)_i(BG)$ and $\beta\in K(n)_j(BG)$. Recall that the intersection product  $\alpha\cap\beta\in K(n)_{i+j}(BG)$ is defined by
$$\alpha\cap\beta=(\lambda_G)_{\ast}^{-1}((\lambda_G)_{\ast}(\aaa)\cup(\lambda_G)_{\ast}(\bbb)).$$
By K\"unneth theorem, $\alpha$ and $\beta$ define an element $\alpha\otimes\beta\in K(n)_{i+j}(BG^2)$. It is related to $\alpha\cap\beta$ by the formula below.

\begin{proposition}\label{capformula}
$\alpha\cap\beta=\Delta^!(\alpha\otimes\beta)$.
\end{proposition}

\begin{proof}
By proposition \ref{prop:iidualcomm} and definitions of $\lambda_G$ and $\lambda_{G\times G}$, we have the commutative diagram
\[
\xymatrix{
K(n)_{\ast}(BG)\ar[rr]^{(\lambda_G)_{\ast}}&&K(n)^{-\ast}(BG)\\
K(n)_{\ast}(BG\times BG) \ar[u]^{\Delta^!}\ar[rr]^{(\lambda_{G\times G})_{\ast}}&&K(n)^{-\ast}(BG\times BG)\ar[u]^{\Delta^*}\\
K(n)_{\ast}(BG)\otimes K(n)_{\ast}(BG) \ar[u]^{\cong}\ar[rr]^{(\lambda_G)_{\ast}\otimes(\lambda_G)_{\ast}}&&K(n)^{-\ast}(BG)\otimes K(n)^{-\ast}(BG)\ar[u]^{\cong}
}\]
By the fact that the right vertical composite in the diagram defines cup product, the formula we want to prove follows from the commutativity of the diagram.
\end{proof}

\begin{lemma}\label{lem:iicap}
Let $G$ be a finite group and $i:H\subset G$ be a subgroup. Then the composite  $$K(n)_{\ast}(BG)\xrightarrow{i^!}K(n)_{\ast}(BH)\xrightarrow{i_{\ast}}K(n)_{\ast}(BG)$$
sends $\aaa\in K(n)_{\ast}(BG)$ to $i_{\ast}i^!(\aaa)=\aaa\cap i_{\ast}([BH])$.
\end{lemma}

\begin{proof}
By taking $K(n)$-homology of (\ref{mackeyGH}) and applying proposition \ref{capformula}, there is a commutative diagram
\[
\xymatrix{
K(n)_{\ast}(BH)\ar[r]^{i_{\ast}}&K(n)_{\ast}(BG)\\
K(n)_{\ast}(BH)\times K(n)_{\ast}(BH)\ar[u]^{\cap_H}&\\
K(n)_{\ast}(BG)\times K(n)_{\ast}(BH)\ar[u]^{i^!\otimes 1}\ar[r]^{1\times i_{\ast}}&K(n)_{\ast}(BG)\times K(n)_{\ast}(BG)\ar[uu]^{\cap_G}
}
\]
Here we add subscripts to $\cap$ to distinguish between the two intersection products in $K(n)_{\ast}(BG)$ and $K(n)_{\ast}(BH)$. The commutative diagram implies that
$$
\alpha\cap_G i_{\ast}([BH])= i_{\ast}(i^!(\aaa)\cap_H[BH])=i_{\ast}i^!(\aaa).
$$
\end{proof}

The intersection product formula $[BH]\cap[BK]=[B(H\cap K)]$ for transverse subgroups $H,K$ of $G$ follows easily from the results above. 

\begin{proof}[Proof of theorem \ref{thm:transversecap}]
By proposition \ref{prop:transversepbmackey}, there is a commutative diagram
$$\xymatrix{
K(n)_{\ast}(B(H\cap K))\ar[r]^-{p_{\ast}}& K(n)_{\ast}(BK)\\
K(n)_{\ast}(BH)\ar[r]^{i_{\ast}}\ar[u]^{q^!}& K(n)_{\ast}(BG)\ar[u]^{j^!}}$$
Together with proposition \ref{prop:iidualcomm} and lemma \ref{lem:iicap}, the commutativity of the diagram implies that
$$j_{\ast}p_{\ast}([B(H\cap K)])=j_{\ast}p_{\ast} q^!([BH])=j_{\ast}j^!i_{\ast}([BH])=i_{\ast}([BH])\cap j_{\ast}([BK]).$$
\end{proof}

\subsection{Cyclic groups}\label{sec:cyc}
We will look at the $K(n)$-intersection product of some simple cyclic $p$-groups. We first recall the computation of $K(n)^{\ast}(B\bZ/p^k)$. Let $\SSS^1$ acts on $\mathbb{C}$ by 
$\theta\cdot z=\theta^{p^k}z$ for $\theta\in\SSS^1\subset \mathbb{C}^*$ and $z\in\mathbb{C}$.
The projection map of $\bC$ induces a vector bundle $\gamma:E\SSS^1\times_{\SSS^1}\bC\to E\SSS^1\times_{\SSS^1}pt$. By taking $E\SSS^1$ to be $\SSS^{\infty}$ with the standard $\SSS^1$-action, $\gamma$ is isomorphic to the $p^k$-fold tensor product of the tautological line bundle over $\bC\mathbb{P}^{\infty}$. Since $K(n)$ is a complex oriented cohomology theory, $K(n)^{\ast}(\bC\mathbb{P}^{\infty})=K(n)^{\ast}\llbracket x \rrbracket$. Using the fact that the $p$-series of the formal group law for $K(n)$ is $[p](x)=v_nx^{p^n}$, the $K(n)$-Euler class of the vector bundle $\gamma$ is $[p^k](x)=v_n^mx^{p^{kn}}$, where $m=(p^{kn}-1)/(p^n-1)$. Note that the sphere bundle of $\gamma$ is $B\bZ/p^k$. Hence, there is a Gysin sequence
\[
\ldots \to K(n)^{*-2}(\bC\mathbb{P}^{\infty})\xrightarrow{\cdot v_n^mx^{p^{kn}}} K(n)^{\ast}(\bC\mathbb{P}^{\infty}) \to K(n)^*(BG) \to \ldots.
\]
Multiplication by the Euler class $v_n^mx^{p^{kn}}$ gives a monomorphism from $K(n)^{\ast}(\bC\mathbb{P}^{\infty})\cong K(n)^{\ast}\llbracket x\rrbracket$ to itself . Hence, $$K(n)^{\ast}(BG)=K(n)^{\ast}\llbracket x\rrbracket/(v_n^mx^{p^{kn}})\cong K(n)^{\ast}[x]/(x^{p^{kn}}),$$ 
which has a basis $1,x,x^2,\ldots x^{p^{kn}-1}$ as a free $K(n)^{\ast}$-module. Therefore, the homology $K(n)_{\ast}(BG)$ is a free $K(n)_{\ast}$-module with the dual basis $b_0,b_1,\ldots,b_{p^{kn}-1}$, where $b_i\in K(n)_{2i}(BG)$.

The intersection product and $K(n)$-fundamental class of $B\bZ/p$ is given in the following proposition.

\begin{proposition}\label{prop:zmodp}
For $n=1$, the intersection product in $K(1)_{\ast}(B\bZ/p)$ is given by
\[
b_i\cap b_j= \begin{cases} 
v_1b_{p-1}-v_1^2b_0 &\mbox{if } i=j=p-1; \\
v_1 b_{i+j-(p-1)} &\mbox{if } 2p-2 > i+j \geq p-1; \\
0 & \mbox{if } i+j < p-1. \end{cases}
\]
The $K(1)$-fundamental class of $B\bZ/p$ is $v_1^{-1}b_{p-1}+b_0\in K(1)_0(B\bZ/p)$.

For $n\geq 2$, the intersection product in $K(n)_{\ast}(B\bZ/p)$ is given by
\[b_i\cap b_j= \begin{cases} v_nb_{i+j-(p^n-1)} &\mbox{if } i+j \geq p^n-1; \\
0 & \mbox{if } i+j < p^n-1. \end{cases}\]
The $K(n)$-fundamental class of $B\bZ/p$ is $v_n^{-1}b_{p^n-1}\in K(n)_0(B\bZ/p)$.
\end{proposition}
\begin{proof}
Let $G=\bZ/p$ and $n$ be a positive integer. By proposition \ref{capformula}, the intersection product in $K(n)_\ast(BG)$ can be computed using the transfer map $\Delta^!:\Sigma^{\infty}BG^2_+\to \Sigma^{\infty}BG_+$ induced by the diagonal map $\Delta:G\to G^2$. Let $F$ denote the formal group law for $K(n)$. Based on a calculation of the transfer map $MU^{\ast}(EG)\to MU^{\ast}(BG)$ in complex cobordism \cite{Q}, it was shown in the proof of proposition 9.2 in \cite{Strickland} that 
$$(\Delta^!)^{\ast}:K(n)^{\ast}(BG) \to K(n)^{\ast}(BG^2)\cong K(n)^{\ast}[x_1,x_2]/(x_1^{p^n},x_2^{p^n})$$ 
sends $1$ to $(\Delta^!)^{\ast}(1)=v_n(x_1-_Fx_2)^{p^n-1}$. 
By lemma 2.1 of \cite{BakuPriddy}, the formal group law for $K(n)$ is given by
\[
\begin{split}
x_1+_Fx_2 = x_1+x_2-v_n\sum_{i=1}^{p-1}\frac{1}{p}\binom{p}{i}&x_1^{i(p^{n-1})}x_2^{(p-i)(p^{n-1})}\\&\quad + \text { terms of degree} \geq p^{2(n-1)}, 
\end{split}
\]
which can be expressed as
\[
\begin{split}
x_1+x_2-v_n\left(\frac{(x_1+x_2)^{p^n}-x_1^{p^n}-x_2^{p^n}}{p}\right)+ \text { terms of degree} \geq p^{2(n-1)}.
\end{split}
\]
It follows that
\[
\begin{split}
x_1-_Fx_2 = x_1-x_2-v_n&\left(\frac{(x_1-x_2)^{p^n}-x_1^{p^n}-(-x_2)^{p^n}}{p}\right)\\
&+ \text { terms of degree} \geq p^{2(n-1)}.
\end{split}
\]
Since $x_1^{p^n}=x_2^{p^n}=0\in K(n)^{\ast}(BG^2)$, any monomials of degree greater than $2(p^n-1)$ in $x_1,x_2$ vanish in $K(n)^{\ast}(BG^2)$. Hence
\begin{align}
(\Delta^!)^{\ast}(1)=&v_n(x_1-_Fx_2)^{p^n-1} \nonumber\\ 
=&v_n(x_1-x_2)^{p^n-1}\nonumber\\
&-v_n^2(p^n-1)(x_1-x_2)^{p^n-2}\left(\frac{(x_1-x_2)^{p^n}-x_1^{p^n}-(-x_2)^{p^n}}{p}\right)\label{transfer1}
\end{align}
Since $\binom{p^n-1}{i}\equiv (-1)^i \pmod p$, the first term of (\ref{transfer1}) is $$v_n\sum_{i=0}^{p^n-1}\binom{p^n-1}{i}x_1^i(-x_2)^{p^n-1-i}=v_n\sum_{i=0}^{p^n-1}x_1^ix_2^{p^n-1-i}.$$ The second term of (\ref{transfer1}) is homogeneous of degree $2p^n-2$ in $x_1,x_2$. The only non-zero monomial of this degree in $K(n)^{\ast}(BG^2)$ is $x_1^{p^n-1}x_2^{p^n-1}$, which has coefficient 
\begin{align*}
&-v_n^2(p^n-1)(-1)^{p^n-1}\left(\left.\binom{2p^n-2}{p^n-1}\right/p\right)
= -v_n^2p^{n-1} 
\end{align*}
in (\ref{transfer1}). Hence,
\[(\Delta^!)^{\ast}(1)= \begin{cases} v_1\sum_{i=0}^{p-1}x_1^ix_2^{p-1-i}-v_1^2x_1^{p-1}x_2^{p-1} &\mbox{if } n = 1; \\
v_n\sum_{i=0}^{p^n-1}x_1^ix_2^{p^n-1-i} & \mbox{if } n \geq 2. \end{cases}\]
The cohomology ring $K(n)^{\ast}(BG)$ can be regarded as a  $K(n)^{\ast}(BG^2)$-module via the pullback $\Delta^{\ast}:K(n)^{\ast}(BG^2)\to K(n)^{\ast}(BG)$. It is a fact that with this module structure, $(\Delta^!)^{\ast}:K(n)^{\ast}(BG)\to K(n)^{\ast}(BG^2)$ is a map of $K(n)^{\ast}(BG^2)$-module. Hence, for $1 \leq j \leq p^n-1$,
\[(\Delta^!)^{\ast}(x^j)=(\Delta^!)^{\ast}(\Delta^{\ast}(x_1^j))=x_1^j(\Delta^!)^{\ast}(1)=
v_n\sum_{i=0}^{p^n-j-1}x_1^{j+i}x_2^{p^n-1-i}.
\]
The formulas of the intersection product in $K(n)_{\ast}(BG)$ given in the proposition follow from proposition \ref{capformula} and the computations of $(\Delta^!)^{\ast}$ above. The $K(n)$-fundamental class can be found by calculating the unit in the intersection product. 
\end{proof}

We next look at the intersection product in $K(n)(B\bZ/p^2)$.

\begin{example}
Let $G=\bZ/p^2$. It has $K(n)$-cohomology $K(n)^{\ast}(BG)\cong K(n)^{\ast}[x]/(x^{p^{2n}})$. The diagonal map $\Delta:G\to G^2$ induces a transfer map
$(\Delta^!)^{\ast}:K(n)^{\ast}(BG) \to K(n)^{\ast}(BG^2)\cong K(n)^{\ast}[x_1,x_2]/(x_1^{p^{2n}},x_2^{p^{2n}}).$
For any $0 \leq \aaa,\bbb,\gamma \leq p^{2n}-1$ with $\aaa+\bbb\geq p^{2n}$,
$$
x_1^{\aaa}x_2^{\bbb}(\Delta^!)^{\ast}(x^{\gamma})=(\Delta^!)^{\ast}(\Delta^{\ast}(x_1^{\aaa}x_2^{\bbb}) x^{\gamma})=(\Delta^!)^{\ast}(x^{\alpha+\beta+\gamma})=0. 
$$
It implies that when expressing $(\Delta^!)^{\ast}(x^i)$ as a linear combination of the $K(n)^{\ast}$-basis $x_1^ix_2^j,0 \leq i,j\leq p^{2n}-1,$ of $K(n)^{\ast}(BG^2)$, the coefficients of those $x_1^ix_2^j$ with $i+j\leq p^{2n}-2$ are zero. Hence, by proposition \ref{capformula}, for $b_i,b_j\in K(n)_{\ast}(BG)$ with $i+j\leq p^{2n}-2$, the intersection product $b_i\cap b_j=0$. 

Let $H=\bZ/p$. By proposition \ref{prop:zmodp}, the $K(n)$-fundamental class of $BH$ is a linear combination of $b_0$ and $v_n^{-1}b_{p^n-1}\in K(n)_0(BH)$. Denote by $[BH]$ its image in $K(n)_0(BG)$ induced by the standard inclusion $i:H\to G$. By our computation above, $[BH]\cap[BH]=0$ in $K(n)_{0}(BG)$. Geometrically, the vanishing of $[BH]\cap[BH]$ can be explained by the following pullback diagram
\[
\xymatrix{
\coprod_p BH \ar[r]\ar[d] & BH\ar^i[d]\\
BH\ar^i[r] & BG.
}\]
The pullback $BH\times_{BG}BH=\coprod_p BH$ is the disjoint union of $p$ copies of $BH$. The intersection product $[BH]\cap[BH]$ is equal to the image of the $K(n)$-fundamental class of $\coprod_p BH$ in $K(n)_{\ast}(BG)$. However, since $[\coprod_p BH]=p[BH]$, it is equal to zero in $K(n)_{\ast}(BG).$  
\end{example}

%\begin{appendix}

%\onlinesignature
\end{document}